\documentclass[10pt,reqno]{amsart}
\usepackage[utf8]{inputenc}
\usepackage[margin=1in]{geometry}
\usepackage{amsfonts,amssymb,amsmath,amsthm,tikz,comment,mathtools,setspace,stmaryrd,enumitem,bbm,mathrsfs}
\usetikzlibrary{arrows.meta, decorations.pathreplacing}
\usetikzlibrary{calc}
\usepackage[export]{adjustbox}
\usepackage[hidelinks]{hyperref}
\numberwithin{equation}{section}
\usepackage[normalem]{ulem}

\allowdisplaybreaks

\newcommand{\N}{\mathbb{N}}
\newcommand{\R}{\mathbb{R}}
\newcommand{\Z}{\mathbb{Z}}
\newcommand{\ZN}{\Z_N}
\newcommand{\Id}{\text{Id}}

\newcommand{\deq}{\overset{d}{=}}
\newcommand{\f}{\frac}
\newcommand{\mbf}{\mathbf}

\newcommand{\Prob}{\mathbb{P}}
\newcommand{\diag}{\operatorname{diag}}
\newcommand{\Geom}{\text{Geom}}
\newcommand{\Exp}{\text{Exp}}
\newcommand{\LIG}{\text{Log-Inv-Gamma}}

\newcommand{\vecsum}{\mathfrak s}
\newcommand{\bx}{\mathbf{x}}
\newcommand{\zD}{\widetilde{D}}
\newcommand{\zT}{\widetilde{T}}

\newcommand{\jZ}{j \in \ZN}
\newcommand{\iZ}{i \in \ZN}
\newcommand{\zP}{\widetilde{\mathcal{P}}}
\newcommand{\Pcal}{\mathcal{P}}
\newcommand{\be}{\begin{equation}}
\newcommand{\ee}{\end{equation}}

\newtheorem{theorem}{Theorem}[section]
\newtheorem{proposition}[theorem]{Proposition}

\newtheorem{lemma}[theorem]{Lemma}
\theoremstyle{definition}
\newtheorem{definition}[theorem]{Definition}

\theoremstyle{remark}
\newtheorem{remark}[theorem]{Remark}

\RequirePackage{color}
\definecolor{darkblue}{rgb}{0.0,0.0,0.7}
\definecolor{darkred}{rgb}{0.5,0.0,0.0}
\definecolor{darkgreen}{rgb}{0.0,0.5,0.0}
\definecolor{indigo}{rgb}{0.3,0,0.5}
\definecolor{byzantium}{rgb}{0.44, 0.16, 0.39}

\title[Discrete periodic Pitman transform]{The discrete periodic Pitman transform: invariances, braid relations, and Burke properties}
\author[Engel]{Eva R. Engel}
\address{Eva R. Engel, Princeton University, Mathematics Department,
Princeton, NJ 08544, USA.}
\email{eva.engel@princeton.edu}

\author[Kra-Caskey]{Benjamin Jasper Kra-Caskey}
\address{Benjamin Jasper Kra-Caskey,  Columbia University, Mathematics Department,
New York, NY 10027, USA.}
\email{bjk2168@columbia.edu}

\author[Lazorenko]{Oleksandr Lazorenko}
\address{Oleksandr Lazorenko,  Columbia University, Mathematics Department,
New York, NY 10027, USA.}
\email{ol2267@columbia.edu}

\author[Maia de Olivera]{Caio Hermano Maia de Oliveira}
\address{Caio Hermano Maia de Oliveira,  Columbia University, Mathematics Department,
New York, NY 10027, USA.}
\email{cm4422@columbia.edu}

\author[Sorensen]{Evan Sorensen}
\address{Evan Sorensen,  Columbia University, Mathematics Department,
New York, NY 10027, USA.}
\email{elsorens@iu.edu}

\author[Wong]{Ivan Wong}
\address{Ivan Wong,  Columbia University, Mathematics Department,
New York, NY 10027, USA.}
\email{iw2261@columbia.edu}

\author[Xu]{Ryan Xu}
\address{Ryan Xu,  Columbia University, Mathematics Department,
New York, NY 10027, USA.}
\email{kx2186@columbia.edu}

\author[Zhang]{Xinyi Zhang}
\address{Xinyi Zhang,  Columbia University, Mathematics Department,
New York, NY 10027, USA.}
\email{xz3272@columbia.edu}

\begin{document}

\begin{abstract}
We develop the theory of the discrete periodic Pitman transform, first introduced by Corwin, Gu, and the fifth author. We prove that the discrete periodic Pitman transform satisfies the same braid relations that are satisfied for the full-line Pitman transform shown by Biane, Bougerol, and O'Connell. This defines a group action of the infinite symmetric group on sequences of vectors in $\mathbb R^{\mathbb Z_N}$.  We prove that, for polymers in a periodic environment, single-path and multi-path partition functions are preserved under the action of this transform on the weights in the polymer model.  Combined with a new inhomogeneous Burke property for the periodic Pitman transform, we prove a multi-path invariance result for the periodic inverse-gamma polymer under permutations of the column parameters. In the limit to the full-line case, we obtain a multi-path extension of a recent invariance result of Bates, Emrah, Martin, Sepp\"al\"ainen, and the fifth author, in both positive and zero-temperature. 
\end{abstract}

\maketitle

\setcounter{tocdepth}{1}
\tableofcontents

\section{Introduction}

The study of the Pitman transform traces back to Pitman's celebrated $2M-X$ theorem \cite{pitman1975one}. Since then, there have been numerous studies around this transform and its many variants. In the semi-discrete case, the positive-temperature lifting of the full-line Pitman transform was introduced by \cite{matsumoto2000analogue,o2002random} in the context of Brownian queuing process and later by \cite{o2012directed} in the setting of semi-discrete directed polymers. In the fully discrete case, the discrete geometric Pitman transform was connected to the geometric lifting of the classical Robinson-Schensted-Knuth (RSK) correspondence \cite{kirillov2001introduction} and was studied in \cite{noumi2002tropicalrobinsonschenstedknuthcorrespondencebirational,COSZ-2014,OSZ-2014,corwin2020invariancepolymerpartitionfunctions} in the context of discrete directed polymer models.

Despite arising in different contexts, these Pitman transforms share a rich set of algebraic and probabilistic properties that have enabled significant progress in the study of stochastic integrable systems, including directed percolation, random polymers, and the KPZ equation. For example, the Burke property--the analogous product measure preservation property of Burke’s theorem for $M/M/1$ queues \cite{burke56}--has been proved in several contexts (see \cite{hsu76,o2001brownian,oconnell02,KOR2002,Martinbatch,DMO,aldo-diac95} as well as \cite{Emrah-Janjigian-Seppalainen-2024,EEEJT} in an inhomogeneous setting) and has been used to compute the fluctuation exponents for several last-passage percolation and polymer models \cite{Cator-Groeneboom-2006,Balazs-Cator-Seppalainen-2006,Seppalainen-Valko-2010,Sepp_l_inen_2012,Emrah-Janjigian-Seppalainen-2023,Landon-Sosoe-Noack-2023}.  Moreover, the joint invariance of last-passage values and polymer partition functions under the zero-temperature and positive-temperature Pitman transforms provides key tools for constructing central limiting objects in the KPZ universality class \cite{DOV} and for proving convergence of many different models to these limiting objects \cite{dauvergne2021scaling,wu2023kpz,Aggarwal-Corwin-Hegde-2024,zhang2025convergence}.

The purpose of this paper is to demonstrate the rich algebraic and probabilistic structure of a periodic version of the Pitman transform, recently introduced in \cite{corwin2024periodicpitmantransformsjointly}. Our first main result, Theorem \ref{thm:braid} proves that the periodic Pitman transform is an involution and proves braid relations for these transforms, analogous to the relations enjoyed by the full-line Pitman transform \cite{biane2005littelmannpathsbrownianpaths} (see also \cite{Dauvergne-Nica-Virag-2022,dauvergne-virag-24}). More specifically, we consider a collection of operators on the space $(\R^{\ZN})^\Z$ of bi-infinite sequences of $N$-periodic vectors and operators $\Pcal_k$, which replace the $k$th and $k +1$st vector with their Pitman transform. These relations show that this collection of operators generates a copy of the infinite symmetric group, which defines a group action on the space of infinite sequences of vectors in $\R^{\Z_N}$.

Our next main result, Theorem \ref{thm:polymer_invariance}, shows that single-path and multi-path partition functions for directed polymers in a periodic environment are preserved under this group action on the columns of weights for the polymer. Previous results of this flavor in a non-periodic setting have appeared in several contexts and are discussed after the statement of the theorem. The key machinery is the encoding of geometric RSK in terms of matrix relations from \cite{noumi2002tropicalrobinsonschenstedknuthcorrespondencebirational}. Corwin \cite{corwin2020invariancepolymerpartitionfunctions} previously used these tools to prove the analogous invariance for full-space polymers. 

We then utilize the invariance of partition functions and a new inhomogeneous Burke property (Proposition \ref{burke-pos-temp}) to prove a distributional invariance of the inverse-gamma polymer partition function under the permutation of the row and column parameters in both the periodic and full-space cases (Theorems \ref{thm:per_permutation-invariance} and \ref{thm:full_per_invariance}). In the full-line case, this gives a multi-path and positive-temperature extension of a recent result of Bates, Emrah, Martin, Sepp\"al\"ainen, and the fifth author \cite{EEEJT}. This result bears similarity to an invariance found previously by Dauvergne \cite{Dauvergne-2022}. See the discussion before and after Theorem \ref{thm:full_per_invariance} for more on the connection to other works. Each of our main results has an analogue in zero-temperature. These are stated in Section \ref{sec:zero_temp_intro}.

\subsection{The positive-temperature discrete periodic Pitman transform}

For $N \in \N$, let $\R^{\ZN}$ denote the set of $N$-tuples of real numbers, together with periodic indexing. That is, for $\mbf X \in \R^{\ZN}$, we may write $\mbf X = (X_1,X_2,\ldots,X_{N})$, where the indices represent elements of $\ZN$ with addition modulo $N$. For a vector $\mbf X\in \R^{\ZN}$ and indices $i,\jZ$, we set 
\be\label{eq:index-periodic-sum}
X_{(i,j]} := \sum_{\ell = i+1}^j X_\ell,\quad\text{and}\quad  X_{[i,j]} := \sum_{\ell = i}^j X_\ell,
\ee
where the sums are taken in cyclic order  (see Figure \ref{fig:per_index}). We follow the convention that $X_{(i,i]} = 0$. 

\begin{figure}[ht]
\centering
\begin{tikzpicture}[scale = 0.5]
        \draw (0,0) rectangle (8,1); 
        \draw (1,1) -- (1, 0);
        \draw (2,1) -- (2, 0);
        \draw (3,1) -- (3, 0);
        \draw (4,1) -- (4, 0);
        \draw (5,1) -- (5, 0);
        \draw (6,1) -- (6, 0);
        \draw (7,1) -- (7, 0);
        \draw[fill = gray] (0,0) rectangle (1,1); 
        \draw[fill = gray] (1,0) rectangle (2,1);
        \draw[fill = gray] (5,0) rectangle (6,1);  
        \draw[fill = gray] (6,0) rectangle (7,1);  
        \draw[fill = gray] (7,0) rectangle (8,1);  
        \node at (1.5,1.75) {\textbf{j]}};
        \node at (5.5,1.75) {\textbf{[i}};
        \draw[->, thick] (0, 1.75) -- (1, 1.75);
        \draw[->, thick] (6, 1.75) -- (8, 1.75);
    \end{tikzpicture}
      \caption{Representation of indices $[i,j]$, demonstrating periodicity.}
      \label{fig:per_index}
\end{figure}

The \textit{positive-temperature discrete periodic Pitman transform} is the map $\Pcal:\R^{\ZN} \times \R^{\ZN} \to \R^{\ZN} \times \R^{\ZN}$ defined by
\be\label{eq:pitman-transform}
\begin{aligned}
\Pcal(\mbf{X}_1,\mbf{X}_2) &:= (T(\mbf{X}_1,\mbf{X}_2),D(\mbf{X}_1,\mbf{X}_2)),\quad\text{where} \\
T_i(\mbf X_1,\mbf X_2) &:= X_{2,i-1} + \log \Biggl(\f{\sum_{\jZ} e^{Y_{[i-1,j]}}}{\sum_{\jZ}e^{Y_{[i,j]}}}   \Biggr),\quad\text{and} \\
D_i(\mbf X_1,\mbf X_2) &:= X_{1,i+1} + \log\Biggl(\f{\sum_{\jZ} e^{Y_{(i,j]}}}{\sum_{\jZ}e^{Y_{(i-1,j]}}}\Biggr),\quad\text{where}\quad Y_\ell = X_{1,\ell+1} - X_{2,\ell}.
\end{aligned}
\ee
This map is a shifted version of the map studied in \cite{corwin2024periodicpitmantransformsjointly} with $\mbf X_1$ and $\mbf X_2$ swapped (see Appendix \ref{sec:DT-basic-properties}).  The swapping of the order of $\mbf X_1$ and $\mbf X_2$ is done in this setting so that the map $\mathcal P$ is an involution (see Theorem \ref{thm:braid}\ref{it:involution} below). By comparison, the map in \cite{corwin2024periodicpitmantransformsjointly} is not an involution; its inverse is the map obtained by first swapping the order of the coordinates, then applying the Pitman transform, then swapping the order again (see \cite[Corollary 5.9]{corwin2024periodicpitmantransformsjointly}).

Note that the maps $\Pcal,T,D$ depend on the length $N$ of the periodicity; we will subsume dependence on this value in  the notation and fix the positive integer $N$, unless otherwise noted.

Let $(\R^{\ZN})^\Z$ be the set of bi-infinite sequences of elements of $\R^{\ZN}$. For $k \in \Z$, we define the operator $\Pcal_k:(\R^{\ZN})^\Z \to (\R^{\ZN})^\Z$ by its action:
\be \label{eq:Pi_def}
\Pcal_k(\ldots, \mbf X_{k-1},\mbf X_k, \mbf X_{k+1},\mbf X_{k+2},\ldots ) = (\ldots,\mbf X_{k-1},T(\mbf X_{k},\mbf X_{k+1}),D(\mbf X_{k},\mbf X_{k+1}),\mbf X_{k+2},\ldots).
\ee
Our first main result is the following:
\begin{theorem}  \label{thm:braid}
    The operators $\Pcal_k$ satisfy the following:
    \begin{enumerate} [label=(\roman*),font=\normalfont]
        \item \label{it:involution} For $k \in \Z$, the map $\Pcal_k$ is an involution.
        \item \label{it:braid} For $k \in \Z$,
        \[
    \Pcal_k \Pcal_{k+1} \Pcal_k = \Pcal_{k+1} \Pcal_k \Pcal_{k+1}.
        \]
        \item \label{it:commute} For $|k - j| > 1$,
        \[
        \Pcal_j \Pcal_k = \Pcal_k \Pcal_j.
        \]
    \end{enumerate}
\end{theorem}
The commutativity in Item \ref{it:commute} of Theorem \ref{thm:braid} is immediate from the definition. Item \ref{it:involution} was previously proved in \cite[Corollary 5.9]{corwin2024periodicpitmantransformsjointly}, although it is stated slightly differently there since the definition of the map is modified (see Appendix \ref{sec:DT-basic-properties}). In the present paper, we give an alternative, self-contained proof of Item \ref{it:involution} in Section \ref{subsec:braid}.  Item \ref{it:braid} is called the braid relation, and it is also proved in Section \ref{subsec:braid}. Crucial to the proof is a matrix relation in Proposition \ref{prop:H_relation}, which is also used to prove Theorem \ref{thm:polymer_invariance}. This matrix relation comes from adapting the matrix representation of geometric RSK from the work of Noumi and Yamada \cite{noumi2002tropicalrobinsonschenstedknuthcorrespondencebirational} to the periodic case.

The braid relation consists of three equalities (\eqref{eq:braid-1}-\eqref{eq:braid-3}), the third of which was previously proved in \cite[Proposition 5.13]{corwin2024periodicpitmantransformsjointly}. From there, we iterate the matrix relation in Proposition \ref{prop:H_relation} in two different ways and use an interpretation of the matrix product as a polymer partition function (Lemma \ref{lem:H_partition}) to obtain the other two equalities. Specifically, Lemma \ref{lem:sum-partition-inv} is the key innovation that allows us to complete the proof.  

 We call a permutation $\sigma:\Z \to \Z$ a \textit{finite permutation} if  $\sigma$ fixes all but finitely many coordinates. Let $\mathcal S_\Z$ denote the group of these finite permutations under composition of functions. The relations satisfied by the maps $\Pcal_k$ in Theorem \ref{thm:braid} are the same as those satisfied by the nearest-neighbor transpositions $(k,k+1)$, which generate the group $\mathcal S_\Z$. Thus, if $\sigma$ is any such finite permutation, if we decompose it in two different ways as
\[
\sigma = (k_1,k_1 + 1)(k_2,k_2 + 1) \cdots (k_n,k_n + 1) = (\ell_1,\ell_1 + 1)(\ell_2,\ell_2 + 1) \cdots (\ell_r,\ell_r + 1),
\]
then Theorem \ref{thm:braid} implies that we have the following equality of operators:
\[
\Pcal_{k_1}\Pcal_{k_2} \cdots \Pcal_{k_n} = \Pcal_{\ell_1}\Pcal_{\ell_2} \cdots \Pcal_{\ell_r}. 
\]
For any finite permutation $\sigma$, we then define the operator $\Pcal_\sigma$ by the product above. The operators $\Pcal_\sigma$ therefore define a group action of $\mathcal S_\Z$ on the space $(\R^{\ZN})^\Z$.

\subsection{Polymers on $\Z^2$ and periodic environments}
    Consider the 2-dimensional integer lattice $\Z^2$. We say that an \textit{up-right path} of length $L\in\N$ is an ordered collection of points
    \be\label{eq:up_right_path}
    \pi:=\left\{(a_i,b_i)_{i\in\llbracket 1,L\rrbracket}\mid(a_{i+1},b_{i+1})-(a_i,b_i)\in\{(1,0),(0,1)\} \ \forall 1\le i<L\right\},
    \ee
    where $(a_i,b_i)\in\Z^2\ \forall i$. Here, the notation $\llbracket a,b \rrbracket$ denotes the set of $k \in \Z$ satisfying $a \le k \le b$. We also define the set of all up-right paths to be $\Pi_{\Z^2}$.

    For a fixed path $\pi$, we say that $(a_1,b_1)$ is the \textit{initial point} of the path while $(a_{L},b_{L})$ is the \textit{end point} of the path. For given start and end points $(i,j),(n,m)\in\Z^2$, we define 
    \be\label{eq:up_right_paths}
    \Pi_{(i,j),(n,m)}:=\{\pi \in\Pi_{\Z^2}\mid (a_1,b_1)=(i,j) \ ,(a_{L},b_{L})=(n,m)\}
    \ee
    to be the set of paths starting at $(i,j)$ and ending at $(n,m)$, noting that $\Pi_{(i,j),(n,m)}\ne\varnothing$ if and only if  $m\ge j $ and $n\ge i$. In this case,  $L=m+n-j-i$ for any $\pi$ in $\Pi_{(i,j),(n,m)}$. We say that two paths $\pi_1,\pi_2\in\Pi_{\Z^2}$ are \textit{non-intersecting} if $\pi_1\cap\pi_2=\varnothing$. We say that a non-intersecting $k$-tuple of paths $(\pi_1,\ldots,\pi_k)$ is called a \textit{multi-path}. See Figure \ref{fig:multi-path} for a visual example.  
    
    As in the case of single paths, we can consider the set of paths between given start and end points. For $k\in\N$ and $k$-tuples of  ordered pairs $U=\bigl((a_i,b_i)\bigr)_{i\in\llbracket1,k\rrbracket}$, $V=\bigl((c_i,d_i)\bigr)_{i\in\llbracket1,k\rrbracket}$, we define the set of multi-paths between them to be  
    \be \label{eq:multipaths}
    \Pi_{U,V}:=\{\pi=(\pi_1,\dots,\pi_k)\mid\pi_i\in\Pi_{(a_i,b_i),(c_i,d_i)} \ \forall i\in\llbracket1,k\rrbracket,\ \pi_i\cap \pi_j=\varnothing \  \forall i\ne j \in \llbracket1,k\rrbracket\}.
    \ee
    We will often require that  there is a unique pairing of initial points in $U$ and end points in $V$ such that a non-intersecting multi-path exists. In other words, we require that if $U$ and $V$ each have $k$ elements, then there does not exist a nontrivial permutation $\varphi$ in  the symmetric group on $k$ elements, $\mathcal S_k$, such that $\Pi_{U,V_\varphi}\ne\varnothing$ where $V_\varphi:=\bigl((c_{\varphi(i)},d_{\varphi(i)})\bigr)_{i\in\llbracket1,k\rrbracket}$. We denote  the pairs of $k$-tuples of ordered pairs satisfying this property as
    \be \label{def:Psi}
    \Psi :=\bigl\{(U,V)\mid U,V\subseteq \Z^2,\ |U| = |V|,\ \Pi_{U,V}\ne\varnothing,\ \text{ and } \Pi_{U,V_{\varphi}}=\varnothing \ \forall \varphi\in \mathcal S_{|U|} \setminus\{\Id\}\bigr\}.
    \ee
    As an example, if $U = \{(0,0),(2,0)\}$ and $V = \{(2,4),(4,4)\}$ (as depicted in Figure \ref{fig:multi-path}), then $(U,V) \in \Psi$. More generally, a pair $(U,V)$ lies in $\Psi$ if and only if $\Pi_{U,V} \neq \varnothing$, the points of $U$ lie along some down-right path, and the points of $V$ lie along a down-right path. See Figure \ref{fig:not_Psi} for an example of a pair $(U,V) \notin \Psi$.
    \begin{figure}
    \centering
     \begin{tikzpicture}[scale=1.5, every node/.style={circle, fill=black, inner sep=1.8pt}]

  \def\cols{4}  
  \def\rows{4}  

  \foreach \x in {0,...,\cols} {
    \foreach \y in {0,...,\rows} {
      \coordinate (P-\x-\y) at (\x, \y);
      \node at (P-\x-\y) {};

      \pgfmathtruncatemacro{\j}{\y}
      \ifnum\x=4
        \node[draw=none, fill=none, anchor=north west, font=\small, scale=0.925] 
          at ([shift={(0.15,0)}]P-\x-\y) {$(\x,\j)$};
      \else
        \node[draw=none, fill=none, anchor=north west, font=\small, scale=0.925] 
          at ([shift={(0.15,0)}]P-\x-\y) {$(\x,\j)$};
      \fi
    }
  }

  \foreach \x in {0,...,\numexpr\cols-1} {
    \foreach \y in {0,...,\rows} {
      \draw[->, line width=.8pt, gray!80] (P-\x-\y) -- (P-\the\numexpr\x+1\relax-\y); 
    }
  }

  \foreach \x in {0,...,\cols} {
    \foreach \y in {0,...,\numexpr\rows-1} {
      \draw[->, line width=.8pt, gray!80] (P-\x-\y) -- (P-\x-\the\numexpr\y+1\relax); 
    }
  }

  \draw[->, line width=1.5pt, blue] (P-0-0) -- (P-0-1); 
  \draw[->, line width=1.5pt, blue] (P-0-1) -- (P-0-2); 
  \draw[->, line width=1.5pt, blue] (P-0-2) -- (P-1-2); 
  \draw[->, line width=1.5pt, blue] (P-1-2) -- (P-2-2); 
  \draw[->, line width=1.5pt, blue] (P-2-2) -- (P-2-3); 
  \draw[->, line width=1.5pt, blue] (P-2-3) -- (P-2-4); 

  \draw[->, line width=1.5pt, red] (P-2-0) -- (P-2-1); 
  \draw[->, line width=1.5pt, red] (P-2-1) -- (P-3-1); 
  \draw[->, line width=1.5pt, red] (P-3-1) -- (P-4-1); 
  \draw[->, line width=1.5pt, red] (P-4-1) -- (P-4-2); 
  \draw[->, line width=1.5pt, red] (P-4-2) -- (P-4-3); 
  \draw[->, line width=1.5pt, red] (P-4-3) -- (P-4-4); 

  \tikzset{ellipsis/.style={inner sep=0pt, fill=none, draw=none, anchor=center, font=\normalsize}}

  \foreach \x in {0,...,\cols} {
    \node[ellipsis] at (\x, \rows + 0.4) {$\vdots$};
  }

  \foreach \x in {0,...,\cols} {
    \node[ellipsis] at (\x, -0.4) {$\vdots$};
  }

  \foreach \y in {4,3,2,1,0} {
    \node[ellipsis] at ({\cols + 0.5}, \y) {$\cdots$};
  }
  \foreach \y in {4,3,2,1,0} {
    \node[ellipsis] at ({-0.5}, \y) {$\cdots$};
  }

    \end{tikzpicture}
    \caption{Example multi-path in $\Z^2$ from $\{(0,0),(2,0)\}$ to $\{(2,4),(4,4)\}$ }
    \label{fig:multi-path}
\end{figure}
\begin{figure}
    \centering
     \begin{tikzpicture}[scale=1.5, every node/.style={circle, fill=black, inner sep=1.8pt}]

  \def\cols{4}
  \def\rows{4}

  \foreach \x in {0,...,\cols} {
    \foreach \y in {0,...,\rows} {
      \coordinate (P-\x-\y) at (\x, \y);
      \node at (P-\x-\y) {};
      \node[draw=none, fill=none, anchor=north west, font=\small, scale=0.925] 
          at ([shift={(0.15,0)}]P-\x-\y) {$(\x,\y)$};
    }
  }

  \foreach \x in {0,...,\numexpr\cols-1} {
    \foreach \y in {0,...,\rows} {
      \draw[->, line width=.8pt, gray!80] 
        (P-\x-\y) -- (P-\the\numexpr\x+1\relax-\y);
    }
  }

  \foreach \x in {0,...,\cols} {
    \foreach \y in {0,...,\numexpr\rows-1} {
      \draw[->, line width=.8pt, gray!80] 
        (P-\x-\y) -- (P-\x-\the\numexpr\y+1\relax);
    }
  }


  \draw[->, line width=1.6pt, red] (P-0-0) -- (P-0-4);
  \draw[->, line width=1.6pt, red] (P-0-4) -- (P-4-4);
  \draw[->, line width=1.6pt, red] (P-1-1) -- (P-1-3);
  \draw[->, line width=1.6pt, red] (P-1-3) -- (P-3-3);


  \draw[->, line width=2.5pt, blue] (0.05,0.05) -- (0.05,2.05);
  \draw[->, line width=2.5pt, blue] (P-0-2) -- (P-3-2);
  \draw[->, line width=2.5pt, blue] (P-3-2) -- (P-3-3);
  \draw[->, line width=2.5pt, blue] (P-1-1) -- (P-4-1);
  \draw[->, line width=2.5pt, blue] (P-4-1) -- (P-4-4);


  \node[circle, fill=green!70!black, draw=black, inner sep=3pt] at (P-0-0) {};
  \node[circle, fill=green!70!black, draw=black, inner sep=3pt] at (P-1-1) {};

  \node[circle, fill=orange, draw=black, inner sep=3pt] at (P-3-3) {};
  \node[circle, fill=orange, draw=black, inner sep=3pt] at (P-4-4) {};

  \node[draw=none, fill=none, above left=3pt] at (P-0-0) {$u_1$};
  \node[draw=none, fill=none, above left=3pt] at (P-1-1) {$u_2$};

  \node[draw=none, fill=none, above right=3pt] at (P-3-3) {$v_1$};
  \node[draw=none, fill=none, above right=3pt] at (P-4-4) {$v_2$};

  \tikzset{ellipsis/.style={inner sep=0pt, fill=none, draw=none, anchor=center, font=\normalsize}}

  \foreach \x in {0,...,\cols} {
    \node[ellipsis] at (\x, \rows + 0.4) {$\vdots$};
    \node[ellipsis] at (\x, -0.4) {$\vdots$};
  }

  \foreach \y in {4,3,2,1,0} {
    \node[ellipsis] at ({\cols + 0.5}, \y) {$\cdots$};
    \node[ellipsis] at ({-0.5}, \y) {$\cdots$};
  }

    \end{tikzpicture}
    \caption{Two multipaths in $\mathbb{Z}^2$ from $U:=\{u_1,u_2\}$ (green) to $V:=\{v_1,v_2\}$ (orange), where $(U,V)\notin \Psi$. To see this, note that there exists a multipath from $(u_1,u_2)$ to $(v_1,v_2)$ (denoted here in blue/thick) as well as a multipath from $(u_1,u_2)$ to $(v_2,v_1)$ (denoted here in red/thin).}
    \label{fig:not_Psi}
\end{figure}

\begin{definition} \label{def:polymer}
  For real-valued weights $\mbf X = \bigl(\mbf X_i\in \R^{\Z}\bigr)_{i \in \Z} = \bigl(X_{(i,j)}\in\R \mid i,j \in \Z\bigr)$ corresponding to the vertices of the directed lattice $\Z^2$, we define the polymer partition function with inverse temperature $\beta > 0$ between the initial point $(i,j)$ and end point $(n,m)$ in $\Z^ 2$ as
    \be\label{eq:point_to_point_partition_function}
    Z^{\beta,\mbf X}(n,m\mid i,j):=\sum_{\pi\in\Pi_{(i,j),(n,m)}}\prod_{(r,\ell)\in\pi}e^{\beta X_{(r,\ell)}}.
    \ee
    The partition function has a natural multi-path extension. For $(U,V)\in \Psi$, we define the \textit{multi-path partition function} between $U$ and $V$ to be
     \be\label{eq:multipath_partition_function}
    Z^{\beta,\mbf X}(V\mid U):=\sum_{\pi\in\Pi_{U, V}}\prod_{(r,\ell)\in\pi}e^{\beta X_{(r,\ell)}}.
    \ee
    We typically take $\beta = 1$, in which case we omit the superscript $\beta$ in the notation.
\end{definition}
\begin{remark}
    In Definition \ref{def:polymer}, we defined a polymer for arbitrary weights on $\Z^2$. In all results below (with the exception of Theorem \ref{thm:full_per_invariance}) we specialize to the case where the weights are periodic in the vertical direction. 
\end{remark}

\subsection{Invariance of periodic polymer partition functions}

Our next main result gives the invariance of the polymer partition functions under the iterated Pitman transform defined in \eqref{eq:pitman-transform}. For $\sigma \in \mathcal S_{\Z}$, we let $\mathcal W^\sigma$ be the set of tuples $(U,V) \in \Psi$, such that
\be \label{eq:Usigma_cond}
\begin{aligned}
\sigma(\Z_{< x}) &= \Z_{< x} \text{ for all }(x,y) \in U, \text{ and }\sigma(\Z_{> x}) = \Z_{> x} \text{ for all }(x,y) \in V.
\end{aligned}
\ee
See Figure \ref{fig:UV_path} for an example.
\begin{figure}
    \centering
    \includegraphics[height=6cm]{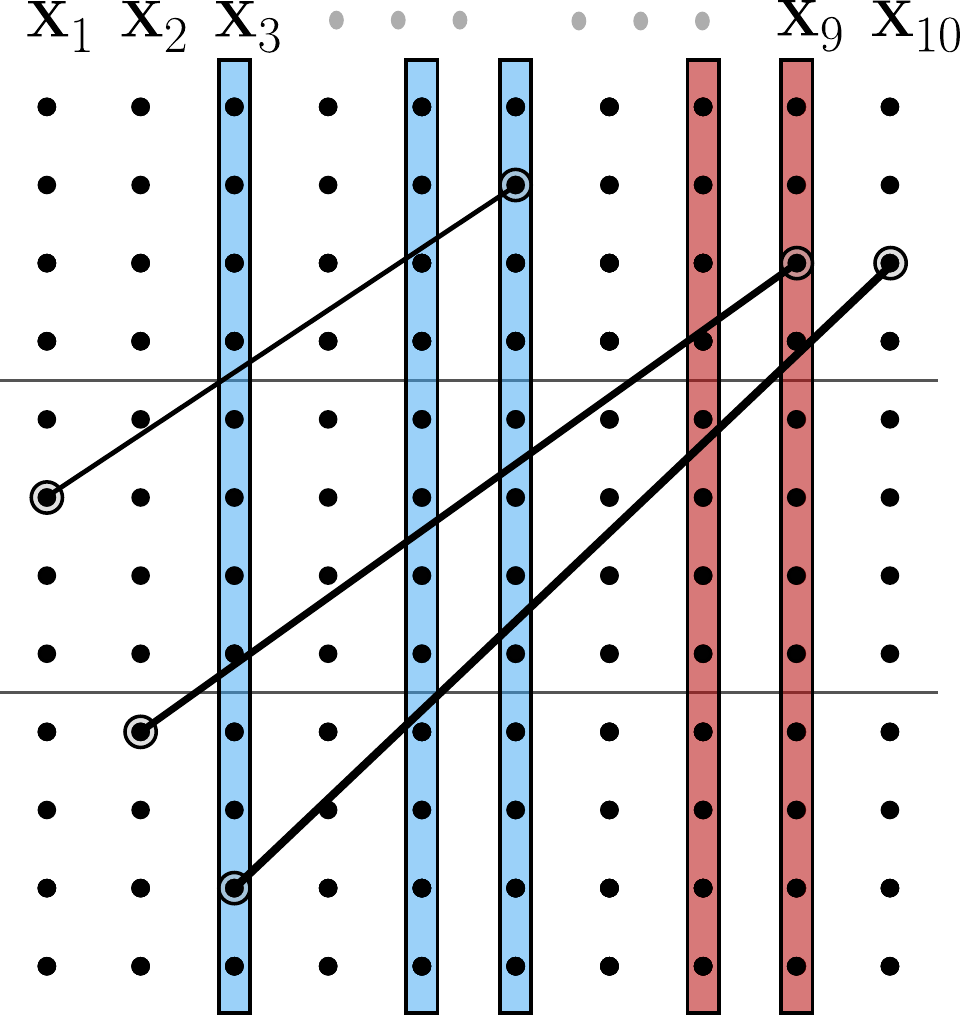}
    \caption{\small An example of an pair $(U,V) \in \mathcal W^\sigma$, where $U$ and $V$ each contain $3$ elements. Each column corresponds to a periodic vector of weights $\mbf X_k$. The gray horizontal lines are used to indicate a periodic environment with $N = 4$. In this figure, the column parameters of columns $3,5$, and $6$ (blue) are permuted, and the column parameters in columns $8$ and $9$ (red) are permuted.  The condition is that, for each of the three end points $(x,y)$ in $U$ (lower left), $\sigma(\Z_{<x}) = \Z_{<x}$, while for each of the three end points $(x,y)$ in $V$ (upper right), $\sigma(\Z_{>x}) = \Z_{>x}$. The multi-path partition function for this pair is preserved after changing the weights from $\mbf X$ to $\sigma \mbf X$, and the distributional equality in Theorem \ref{thm:per_permutation-invariance} holds jointly over all end point  pairs $(U,V) \in \Psi$ satisfying this property. }
    \label{fig:UV_path}
\end{figure}
\begin{theorem} \label{thm:polymer_invariance}
    Let $N \in \N$, and let $\mbf X=(\mbf X_k\in\R^{\ZN} \mid k\in \Z) = (X_{k,\ell}\mid k \in \Z, \ell \in \ZN)$ be a collection of real-valued weights corresponding to the columns of the lattice.  Extend these periodically to weights on $\Z \times \Z$ by the condition that $X_{k,j} = X_{k,\ell}$ whenever $j \equiv \ell \mod N$. Then, if $\sigma \in \mathcal S_\Z$ and $(U,V) \in \mathcal W^\sigma$,
    \be\label{eq:polymer_invariance}
    Z^{\Pcal_\sigma \mbf X}(V\mid U)=Z^{\mbf X}(V\mid U).
    \ee
\end{theorem}

Theorem \ref{thm:polymer_invariance} is proved in Section \ref{sec:polymer_invariance}. It is the first of its kind for a polymer in a periodic environment. The invariance under Pitman transformations in other contexts has been widely studied. The concept can be traced back to queuing theory from the works \cite{Weber1979,TsoWal1987}, where it was shown that the law of the departures process from two $\cdot/M/1$ queues in series is invariant under swapping of the service rates of the two queues. In the polymer context,  Corwin \cite{corwin2020invariancepolymerpartitionfunctions} used the work of Noumi and Yamada \cite{noumi2002tropicalrobinsonschenstedknuthcorrespondencebirational} to prove an analogous result to our Theorem \ref{thm:polymer_invariance} for a full-line version of the discrete Pitman transform. A similar result for a bi-infinite recentered version of the positive-temperature discrete Pitman transform can be found in \cite[Lemma A.5]{Busani-Seppalainen-2022} with a different proof, although it is not explicitly stated as a polymer invariance result there. The proof of Theorem \ref{thm:polymer_invariance} likewise uses the machinery from \cite{noumi2002tropicalrobinsonschenstedknuthcorrespondencebirational}, but we instead work with $\Z \times \Z$ periodic matrices instead of finite-dimensional matrices. The matrix relations in Proposition \ref{prop:H_relation} that are key to proving the braid relations in Theorem \ref{thm:braid}\ref{it:braid} are also the key to the proof of Theorem \ref{thm:polymer_invariance}. In fact, the only need for the braid relations in the proof of Theorem \ref{thm:polymer_invariance} is to ensure that $\mathcal P_\sigma$ is well-defined. The proof follows inductively from the case where $\sigma$ is a transposition, where the braid relations are not needed. Once Theorem \ref{thm:polymer_invariance} is proved in the case of single paths, a Lindstr\"{o}m-Gessel-Viennot argument is used to upgrade this to general endpoint pairs $U$ and $V$.

The braid relations are in fact closely related to the  polymer invariance. In \cite[Lemma A.2]{corwin2024periodicpitmantransformsjointly}, it is shown that $D(\mbf X_1,\mbf X_2)$ is the partition function for a polymer with two columns whose initial endpoint is at $-\infty$ (defined precisely in \cite[Equation (4.3)]{corwin2024periodicpitmantransformsjointly}). In that sense,  Equation \eqref{eq:braid-3} below (which is one of three equations giving the braid relation and was originally proved in \cite[Proposition 5.13]{corwin2024periodicpitmantransformsjointly}) can be used to prove an invariance statement for the single-path polymer under the periodic Pitman transform (when the initial point is at $-\infty$; see \cite[Corollary 5.14]{corwin2024periodicpitmantransformsjointly}). By contrast, the invariance result in Theorem \ref{thm:polymer_invariance} holds for point-to-point polymers with finite endpoints. Some of the intermediate equalities in the proof of \cite[Corollary 5.14]{corwin2024periodicpitmantransformsjointly} could be used as an alternative starting point for a proof of the point-to-point  polymer invariance of Theorem \ref{thm:polymer_invariance}, but the proofs of these equalities are quite long and technical. The technique of the present paper highlights how fundamental identities of the periodic Pitman transform (specifically Lemmas \ref{lem:DT_ordering}-\ref{lem:DT_identities}) along with the Noumi-Yamada framework can be used to give a clean proof of invariance. 

Theorem \ref{thm:polymer_invariance} limits to a zero-temperature analogue in Theorem \ref{thm:LPP_invar}. This zero-temperature limit is the periodic analogue of invariances shown in \cite{Fan_2020,DOV,Dauvergne-Nica-Virag-2022,dauvergne-virag-24,EEEJT}. Specifically, \cite[Lemma 4.3]{Fan_2020} showed invariance of the single-path passage times under a full-line, bi-infinite and recentered version of the Pitman transform, and the extension to the multipath case can then be derived from the definition of the Pitman transform (see, for example, \cite[Remark 4.3]{EEEJT}). The semi-discrete analogue of this invariance proved as \cite[Proposition 4.1]{DOV}  was key in constructing the directed landscape--the central limiting object of the KPZ universality class.

\subsection{Comparison to the full-line Pitman transform}
Motivated by the last paragraph, it would be interesting to see if Theorem \ref{thm:polymer_invariance} could be the starting point for constructing and proving convergence to the periodic directed landscape. However, we note that the action of this transform is qualitatively different than the classical notion in the full-line case. Following the conventions from \cite{oconnell02,o2003path,o2003conditioned,biane2005littelmannpathsbrownianpaths,DOV}, the full-line continuous Pitman transform is defined as follows: for continuous functions $f,g:[0,\infty) \to \R$, we consider the map $(f,g) \mapsto (f \otimes g,g \odot f)$, where 
\be \label{eq:fg_transform}
\begin{aligned}
(f \otimes g)(t) &= g(t) + \inf_{0 \le s \le t}[f(s) - g(s)],\quad\text{and} \\
(g \odot f)(t) &= f(t) + \sup_{0 \le s \le t}[g(s) - f(s)].
\end{aligned}
\ee
Assuming that $g(0) = f(0) = 0$, we can see from the definition that $(f \otimes g)(t) \le (g\odot f)(t)$ for all $t \ge 0$. Thus, the Pitman transform naturally orders the two functions while preserving their sum; i.e., $(f \otimes g)(t) + (g \odot f)(t) = f(t) + g(t)$. This ordering of functions was essential in obtaining the asymptotics that led to the construction of the directed landscape in \cite{DOV}. The preservation of sums holds for the periodic Pitman transform (see Lemma \ref{lem:DT_identities}\ref{sum-Di-Ti}), but the ordering of the outputs does not hold. To see why this is the case, note that the Burke property shown in \cite[Proposition 5.12]{corwin2024periodicpitmantransformsjointly} (and extended to an inhomogeneous setting in Proposition \ref{burke-pos-temp}) states that if the inputs are independent log-inverse-gamma random walks, then the outputs are also independent log-inverse-gamma random walks. Specifically, the independence implies that the outputs of $\mathcal P$ are not ordered as in the full-line case. 

In \eqref{eq:fg_transform}, one should think of $f$ and $g$ as random walks/Brownian motions, while in our setting, one should think of $\mbf X_1,\mbf X_2$ as vectors of increments of a random walk. From this perspective, the periodic Pitman transform \eqref{eq:pitman-transform} is of a form similar to the full-line Pitman transform \eqref{eq:fg_transform}, but it is in fact more closely aligned with a recentered version of the Pitman transform where the initial point of a random walk is sent to $-\infty$ (see, for example, \cite[Section 5]{Dauvergne-Nica-Virag-2022} and \cite[Equations (8)-(9)]{dauvergne-virag-24}). This recentering is necessary in the periodic setting as there is no natural ``starting point."  More precisely, it is shown in \cite[Appendix F]{corwin2024periodicpitmantransformsjointly} that the periodic Pitman transform can be obtained as a specialization of a Pitman-like queuing map with bi-infinite sequences as inputs used in \cite{Bates-Fan-Seppalainen-2025}. This could potentially give another route to proving the invariance of Theorem \ref{thm:polymer_invariance} by starting with the full-line queuing transform from \cite{Bates-Fan-Seppalainen-2025}, then specializing to periodic input. However, the full-line queuing transform in \cite{Bates-Fan-Seppalainen-2025} is only defined for two input sequences which obey a law of large numbers and are such that the first input has strictly smaller limiting slope than the second.  Hence, extending invariance from the full-line to the periodic case would require an additional analytic continuation argument. By contrast, the periodic Pitman transform is well-defined for any two input sequences, and  we work directly in the periodic setting with a finite number of input variables. The recentering of the Pitman transform also gives rise to the Burke properties shown in Propositions \ref{burke-pos-temp}, \ref{burke_geom}, and \ref{burke_exp}, which do not follow from the analogous full-line results.  As a corollary of Theorem \ref{thm:polymer_invariance} and these Burke properties, we obtain a permutation invariance result for the periodic multi-path inhomogeneous inverse-gamma polymer in Theorem \ref{thm:per_permutation-invariance} below. As $N \to \infty$, we obtain Theorem \ref{thm:full_per_invariance}, which extends the previous full-line result of \cite[Theorem 3.1]{EEEJT} to the positive-temperature and multi-path setting.

\subsection{Permutation invariance}
In this section, we consider an inhomogeneous, periodic inverse-gamma polymer. Let $\lambda > 0$, and let and $\mbf a \in \R^\Z$ and $\mbf b \in \R^{\ZN}$ be sequences of real numbers such that $a_i + b_j > 0$ for $i \in \Z$ and $\jZ$. We take mutually independent weights $\mbf X^{\mbf a,\mbf b} = \bigl(X_{(i,j)}\mid  i \in \Z,\ \jZ\bigr)$ such that $X_{(i,j)} \sim \LIG(a_i + b_j,\lambda)$  for $i \in \Z$ and $\jZ$ (see Definition \ref{def:log-inv-gamma} below). Extend these periodically to weights on $\Z \times\Z$.  

\begin{theorem} \label{thm:per_permutation-invariance}
    For vectors $\mbf a \in \R^{\Z}$ and $\mbf b \in \R^{\ZN}$ satisfying $a_i + b_j > 0$ for $i \in \Z$ and $\jZ$, define periodic weights $\mbf X^{\mbf a,\mbf b}$ as above and extend them to $\Z \times \Z$. Then, for any finite permutation $\sigma \in \mathcal S_{\Z}$, we have 
    \[
    \Bigl(Z^{\mbf X^{\mbf a,\mbf b}}(V \mid U): (U,V) \in \mathcal W^{\sigma}\Bigr) \deq \Bigl(Z^{\mbf X^{\sigma\mbf a,\mbf b}}(V \mid U): (U,V) \in \mathcal W^{\sigma}\Bigr),
    \]
    where $\sigma \mbf a = (a_{\sigma(i)})_{i \in \Z}$. 
\end{theorem}
The condition \eqref{eq:Usigma_cond} ensures that, between the pairs of points in $(U,V)$, the column rates are either permuted among themselves or are left unchanged. See Figure \ref{fig:UV_path}.

We can also permute the $\mbf b$ parameters as well, with doubly periodic weights $(X_{(i,j)} \mid i \in \Z_M,\jZ)$ for $M,N \in \N$, when we restrict to points that lie inside the box $\llbracket 0,M \rrbracket \times \llbracket 0,N \rrbracket$. In this case, the model isn't actually periodic at all; it is simply a restriction of the full-space polymer to a finite grid. For finite permutations $\sigma,\tau\in \mathcal S_{\Z}$, define the set $\mathcal W^{\sigma,\tau}$ to be the set of tuples $(U,V)$, where $U,V$ are each sets of end point pairs in $\Z \times \Z$ satisfying the condition \eqref{eq:Usigma_cond} for $\sigma$ as well as the condition
\be \label{eq:Usigma_tau_cond}
\begin{aligned}
\tau(\Z_{< y}) = \Z_{< y} \text{ for all }(x,y) \in U \text{ and } \tau(\Z_{> y}) = \Z_{> y} \text{ for all }(x,y) \in V.
\end{aligned}
\ee
The following is an extension of the recent result \cite[Theorem 3.1]{EEEJT} to multi-paths and to the positive-temperature case. A sketch of the proof for the extension to multi-paths in zero-temperature was previously discussed in \cite[Remark 4.3]{EEEJT}. The technique follows the same outline as developed in \cite{EEEJT}, as we use a coupling of the weights given by applying a Pitman transform to neighboring rows and columns. However, the coupling we use in this proof is different because we treat the weights as periodic outside a large box and apply the periodic Pitman transform. The algebra involved in this positive-temperature case is also different.
\begin{theorem} \label{thm:full_per_invariance}
    For vectors $\mbf a \in \R^{\Z}$ and $\mbf b \in \R^{\Z}$ satisfying $a_i + b_j > 0$ for $i,j \in \Z$ define independent weights $\mbf X^{\mbf a,\mbf b} = \bigl(X_{(i,j)}: i,j \in \Z\bigr)$ such that $X_{(i,j)} \sim \LIG(a_i + b_j,\lambda)$ for all $i,j$. Then, for any finite permutations $\sigma,\tau \in \mathcal S_{\Z}$, we have 
    \[
    \Bigl(Z^{\mbf X^{\mbf a,\mbf b}}(V \mid U): (U,V) \in \mathcal W^{\sigma,\tau}\Bigr) \deq \Bigl(Z^{\mbf X^{\sigma\mbf a,\tau \mbf b}}(V \mid U): (U,V) \in \mathcal W^{\sigma,\tau}\Bigr),
    \]
    where $\sigma \mbf a = (a_{\sigma(i)})_{i \in \Z}$ and $\tau \mbf b = (b_{\tau(j)})_{j \in \Z}$. 
\end{theorem}

Theorems \ref{thm:per_permutation-invariance} and \ref{thm:full_per_invariance} are proved in Section \ref{sec:Burke_permutation}. When the initial point is fixed, invariances of this flavor can be deduced from the symmetry of Schur functions and Whittaker functions (see, for example, \cite{Boro-Peche-08,COSZ-2014,OSZ-2014,Borodin-Corwin-2014}). A similar type of invariance for last-passage percolation was discovered in the work of Dauvergne \cite[Theorem 1.5]{Dauvergne-2022}, motivated by a shift-invariance result previously proved by Borodin, Gorin, and Wheeler \cite{Borodin-Gorin-Wheeler-2022}. Connections between these results are discussed further in \cite[Section 3]{EEEJT}.

\subsection{Zero-temperature limits} \label{sec:zero_temp_intro}

We define the \textit{zero-temperature discrete periodic Pitman transform} as  the map $\zP:\R^{\ZN} \times \R^{\ZN} \to \R^{\ZN} \times \R^{\ZN} $ defined by
\be \label{eq:zero_temp_Pitman}
\begin{aligned}
\zP(\mbf X_1,\mbf X_2) &:= \bigl(\zT(\mbf X_1,\mbf X_2),\zD(\mbf X_1,\mbf X_2)\bigr),\quad\text{where} \\
\zT_i(\mbf X_1,\mbf X_2) &:= X_{2,i-1} + \max_{\jZ}Y_{[i-1,j]} -\max_{\jZ}Y_{[i,j]}    ,\quad\text{and} \\
\zD_i(\mbf X_1,\mbf X_2) &:= X_{1,i+1} + \max_{\jZ}Y_{(i,j]} - \max_{\jZ} Y_{(i-1,j]}, \quad\text{where}\quad Y_\ell = X_{1,\ell+1} - X_{2,\ell}.
\end{aligned}
\ee
The mapping $\zD$ is similar but not equivalent to a mapping used to describe the invariant measures for multi-type TASEP by Ferrari and Martin \cite{ferrari2007stationary}. See \cite[Appendix F]{corwin2024periodicpitmantransformsjointly} for more on this connection. One immediately sees that we can obtain $\zP$ as a limiting case of the positive-temperature discrete periodic Pitman transform: 
\[
\zP(\mbf X_1,\mbf X_2) = \lim_{\beta \to \infty} \f{1}{\beta} \Pcal(\beta \mbf X_1,\beta \mbf X_2).
\]

\begin{definition}
For real-valued weights $\mbf X = \bigl(X_{(i,j)}\in\R \mid i,j \in \Z\bigr)$ corresponding to the vertices of the directed lattice $\Z^2$, we define the last-passage percolation (LPP) time for both a single pair of points and multi-paths analogously as in \eqref{eq:point_to_point_partition_function}-\eqref{eq:multipath_partition_function}: 
\[
G^{\mbf X}(V \mid U) = \max_{\pi \in \Pi_{U,V}} \sum_{(r,\ell) \in \pi} X_{(r,\ell)}.
\]
\end{definition}

For $k \in \Z$, we can also define the map $\zP_k: (\R^{\ZN})^\Z \to (\R^{\ZN})^\Z $ by \eqref{eq:Pi_def}, with $T,D$ replaced by $\zT,\zD$.

With this notation defined, we get the following zero-temperature analogues of our main theorems.

\begin{theorem} \label{thm:zer_temp_braid}
     The same relations in Theorem \ref{thm:braid} hold for the maps $\zP_k$. 
\end{theorem}
Theorem \ref{thm:zer_temp_braid} similarly allows us to define maps $\zP_\sigma$ for a general finite permutation $\sigma \in \mathcal S_\Z$. 

\begin{theorem} \label{thm:LPP_invar}
    The same invariance in Theorem \ref{thm:polymer_invariance} holds when $Z^{\mbf X}$ is replaced by $G^{\mbf X}$, and the maps $\Pcal_{\sigma}$ are replaced by $\zP_{\sigma}$.
\end{theorem}

\begin{theorem} \label{thm:zer_temp_perm_invar}
    The analogues of Theorems \ref{thm:per_permutation-invariance} and \ref{thm:full_per_invariance}  hold, when $Z^{\mbf X}$ is replaced by $G^{\mbf X}$ and, instead of log-inverse-gamma weights, the weights satisfy either $X_{(i,j)} \sim \Geom(q_ir_j)$ for all $i,j$, where $q_ir_j \in (0,1)$, or $X_{(i,j)} \sim \Exp(a_i + b_j)$ for all $i,j$, where $a_i + b_j > 0$  (see Definitions \ref{def:geom}-\ref{def:exp} below).
\end{theorem}
The details of the limit transitions for the proofs of these theorems are given in Appendix \ref{appx:zero_temp}.

\subsection{Organization of the paper} Section \ref{sec:algebra} develops the algebraic structure of the periodic Pitman transform, culminating in the proof of Theorem \ref{thm:braid} in Section \ref{subsec:braid}.
Section \ref{sec:polymer_invariance} is focused on proving the partition function invariance in Theorem \ref{thm:polymer_invariance}. Section \ref{sec:Burke_permutation} proves  an inhomogeneous Burke property for the periodic Pitman transform and uses it to prove Theorems \ref{thm:per_permutation-invariance} and \ref{thm:full_per_invariance}.  Appendix \ref{sec:DT-basic-properties} contains some prior results from the work \cite{corwin2024periodicpitmantransformsjointly} and some auxiliary technical results. The necessary details for the zero-temperature analogues in Theorems \ref{thm:LPP_invar}-\ref{thm:zer_temp_perm_invar} are proved in Appendix \ref{appx:zero_temp}. 

\subsection{Notation and conventions}
 We use $\deq$ to indicate that two random variables or stochastic processes are equal in distribution. Below, we define our notation for geometric, exponential, and log-inverse-gamma random variables.

\begin{definition}[Geometric Random Variable] \label{def:geom}
A random variable $X \in \Z_{\geq 1}$ has a geometric distribution with parameter $q \in [0,1]$ if for all $n \in \Z_{\geq 0}$, we have that 
$$\Prob(X=n) = q^n(1-q)$$
In this case, we write $X \sim \Geom(q)$.
\end{definition}

\begin{definition}[Exponential Random Variable] \label{def:exp}
A random variable $X \in \R$ has an exponential distribution with rate $\alpha>0$ if
\begin{align*}
\Prob(X \in dx) = 
\begin{cases} 
\alpha e^{-\alpha x} &x \geq 0 \\
0 &x < 0
\end{cases}
\end{align*}
In this case, we write $X \sim \Exp(\alpha)$.
\end{definition}

\begin{definition}[Log-Inverse-Gamma Random Variable]\label{def:log-inv-gamma}
A random variable $X \in \R$ has a log-inverse-gamma distribution with shape $\gamma>0$ and scale $\lambda > 0$ if 
$$\Prob(X \in dx) = \frac{\lambda^{\gamma}}{\Gamma(\gamma)} e^{-\gamma x}e^{-\lambda e^{-x}}$$
In this case, we write $X \sim \LIG(\gamma,\lambda)$.
\end{definition}

\subsection{Acknowledgements}
This work was completed during the Columbia Math Summer Undergraduate Research Program in Summer 2025, and we acknowledge support for this program from the Columbia University math department. We thank George Dragomir for organizing this program. During the program, E.S. and X.Z. were partially supported by Ivan Corwin's Simons Investigator Grant \#929852. In addition, X.Z. was supported by Ivan Corwin's NSF grant DMS:2246576. E.R.E. was supported by Princeton University's Office of Undergraduate Research OURSIP Program through the Frances Lane Summer Research Fund and by Princeton University's Office of Undergraduate Research Undergraduate Fund for Academic Conferences through the Hewlett Foundation Fund. This material is based upon work supported by the National Science Foundation under Grant No. 2015553. We thank Ivan Corwin and Amol Aggarwal for helpful discussions related to this project. We also thank the anonymous referees for helpful comments that have greatly improved the presentation of our results. 

\section{Algebraic Properties of the Pitman Transform} \label{sec:algebra}
The end goal of this section is to prove Theorem \ref{thm:braid}, which is done in in Section \ref{subsec:braid}. We start by developing the necessary algebraic background. Versions of the following two results are proved in \cite{corwin2024periodicpitmantransformsjointly} for shifted versions of the maps $D$ and $T$. We provide self-contained proofs in this section.

For this lemma and below, we define the notation
\[
\vecsum(\mbf X) = \sum_{i \in \Z_N} X_i,\quad \text{for }\mathbf X \in \R^{\Z_N}.
\]
\begin{lemma} \label{lem:DT_ordering}
Let $\mbf X_1, \mbf X_2\in \R^{\ZN}$. Then,
\begin{enumerate}[label=(\roman*),font=\normalfont]
    \item \label{sum-Ti-Di} $\vecsum(T(\mbf X_1, \mbf X_2)) = \vecsum(\mbf X_2)$, and $\vecsum(D(\mbf X_1, \mbf X_2)) = \vecsum(\mbf X_1)$.
    \item \label{sum-zTi-zDi} $\vecsum(\zT(\mbf X_1, \mbf X_2)) = \vecsum(\mbf X_2)$, and $\vecsum(\zD(\mbf X_1, \mbf X_2)) = \vecsum(\mbf X_1)$.
\end{enumerate} 
\end{lemma}
\begin{proof}
These follow immediately from the definitions \eqref{eq:pitman-transform} and \eqref{eq:zero_temp_Pitman}, since the sums are telescoping. 
\end{proof}
\begin{lemma} \label{lem:DT_identities} 
Let $\mbf{X}_1, \mbf X_2\in \R^{\ZN}$. Then, for $\iZ$,
\begin{enumerate}[label=(\roman*), font = \normalfont]
    \item \label{sum-Di-Ti} $D_i(\mbf{X}_1, \mbf X_2) + T_{i}(\mbf{X}_1, \mbf X_2)=X_{1, i} + X_{2,i}$,
    \item \label{sum-eDi-eTi} $e^{-D_i(\mbf{X}_1, \mbf X_2)} + e^{-T_{i+1}(\mbf{X}_1, \mbf X_2)}=e^{-X_{2, i}} + e^{-X_{1, i+1}}$,\quad and
    \item \label{sum-zDi-zTi} $\zD_i(\mbf{X}_1, \mbf X_2) + \zT_i(\mbf{X}_1, \mbf X_2)=X_{1, i} + X_{2,i}$.
\end{enumerate}
\end{lemma}
\begin{proof} 
\textbf{Item \ref{sum-Di-Ti}:} From the definition \eqref{eq:pitman-transform}, we observe that 
\begin{align*}
T_i(\mbf X_1, \mbf X_2)  &= X_{2,i-1} + \log \Biggl(\f{e^{Y_{i-1}}\sum_{\jZ} e^{Y_{(i-1,j]}}}{e^{Y_i} \sum_{\jZ}e^{Y_{(i,j]}}}   \Biggr)  \\
&= X_{2,i-1} + Y_{i-1} - Y_i - \log\Biggl(\f{\sum_{\jZ} e^{Y_{(i,j]}}}{\sum_{\jZ}e^{Y_{(i-1,j]}}}\Biggr) \\
&= X_{2,i} + X_{1,i} - D_i(\mbf X_1, \mbf X_2).
\end{align*}

\medskip \noindent \textbf{Item \ref{sum-eDi-eTi}:} We compute this as follows: 
\begin{align*}
    e^{-D_i(\mbf X_1, \mbf X_2)} + e^{-T_{i+1}(\mbf X_1, \mbf X_2)} &= e^{-X_{1, i+1}} \left( \frac{\sum_{j \in \mathbb{Z}_N}e^{Y_{(i-1, j]}}}{\sum_{j \in \mathbb{Z}_N}{e^{Y_{(i, j]}}}} \right) + e^{-X_{2, i}} \left( \frac{\sum_{j \in \mathbb{Z}_N}e^{Y_{[i+1, j]}}}{\sum_{j \in \mathbb{Z}_N}{e^{Y_{[i, j]}}}} \right)\\
    &= e^{-X_{2, i}} \left( \frac{\sum_{j \in \mathbb{Z}_N}e^{Y_{(i-1, j]}}}{e^{Y_i}\sum_{j \in \mathbb{Z}_N}{e^{Y_{(i, j]}}}} \right) + e^{-X_{1, i+1}} \left( \frac{e^{Y_i} \sum_{j \in \mathbb{Z}_N}e^{Y_{[i+1, j]}}}{\sum_{j \in \mathbb{Z}_N}{e^{Y_{[i, j]}}}} \right) \\
    &= e^{-X_{2, i}} \left( 1 + \frac{1 - e^{\sum_{j \in \mathbb{Z}_N}Y_j}}{\sum_{j \in \mathbb{Z}_N}{e^{Y_{[i, j]}}}} \right) + e^{-X_{1, i+1}} \left( 1 + \frac{e^{Y_i}(e^{\sum_{j \in \mathbb{Z}_N}Y_j} - 1)}{\sum_{j \in \mathbb{Z}_N}{e^{Y_{[i, j]}}}} \right) \\
    &= e^{-X_{2, i}} + e^{-X_{1, i+1}} + e^{-X_{2, i}} \left( \frac{1 - e^{\sum_{j \in \mathbb{Z}_N}Y_j}}{\sum_{j \in \mathbb{Z}_N}{e^{Y_{[i, j]}}}} + \frac{e^{{\sum_{j \in \mathbb{Z}_N}}Y_j} - 1}{\sum_{j \in \mathbb{Z}_N}{e^{Y_{[i, j]}}}} \right) \\
    &= e^{-X_{2, i}} + e^{-X_{1, i+1}}.
\end{align*}

\medskip \noindent \textbf{Item \ref{sum-zDi-zTi}:} This follows as the zero-temperature limit of Item \ref{sum-Di-Ti}.  
\end{proof}

\subsection{Matrix encoding} \label{sec:matrix_encoding}
The proofs of Theorems \ref{thm:braid} and \ref{thm:polymer_invariance} follow by using the algebraic properties above to prove certain matrix relations. For this, we adapt the framework of Noumi and Yamada \cite{noumi2002tropicalrobinsonschenstedknuthcorrespondencebirational} to infinite matrices which have a periodic structure. 

\begin{definition}\label{def:EH_definition}
For $\bx\in\R^\Z$,  we introduce two functions $E:\R^{\Z}\to\R^{\Z\times \Z}$ and $H:\R^{\Z}\to\R^{\Z\times \Z}$ as follows: 
\be \label{eq:EH_definition}
E(\bx):=\sum_{i\in\Z} e^{x_i}E_{i,i}+\sum_{i\in\Z} E_{i,i+1}\,,\quad  H(\bx):=\sum_{i\le j} e^{x_i+x_{i+1}+\ldots+ x_j} E_{i,j} \ee
where $E_{p,q}$ is the $\Z\times\Z$ matrix with a $1$ at index $(p,q)$ and $0$s elsewhere.
\end{definition}

These matrix functions admit natural extensions for $\bx\in\R^{\ZN}$, where, for $k\notin\{1,\ldots, N\}$, we simply set $x_k\coloneqq x_{\ell}$, where $\ell$ is the unique element of $\{1,\ldots,N\}$ satisfying $\ell \equiv k\pmod N$. All of the same properties that hold for $\Z$ trivially also hold for $\ZN$, via the inclusion map $i:\R^{\ZN}\to\R^\Z$.

The intuition behind this approach is that matrix multiplication $H(\mbf x_1)\cdots H(\mbf x_n)$ naturally encodes the partition function across a certain diagram with vertex weights given by the arrays $\mbf x_1,\ldots, \mbf x_n$ as discussed later in Lemma \ref{lem:H_partition}. Meanwhile, the outputs of map $E$ are more algebraically straightforward and closely related to the matrix inverse of map $H$, which we prove in the next Lemma \ref{lem:EH_relation} as an infinite-dimensional analogue of \cite[Equation 1.5]{noumi2002tropicalrobinsonschenstedknuthcorrespondencebirational}.
\begin{lemma}\label{lem:EH_relation}
    For all $\bx \in \R^\Z$, we have
    \be\label{eq:EH_relation}
    \Delta H(\bx)\Delta E(-\bx)=E(-\bx)\Delta H(\bx)\Delta= \mathrm{Id},
    \ee
    where $\Delta=\diag ((-1)^{i-1})_{i\in\Z}$ and $\mathrm{Id}=\diag(1)_{i\in \Z}$.
\end{lemma}
\begin{proof}
We compute $H(\bx) \Delta E(-\bx)$ as follows:
\begin{align*}
    H(\bx) \Delta E(-\bx) &= H(\bx)\left(\sum_{j \in \Z}(-1)^{j-1} e^{-x_j}E_{j,j} + \sum_{j \in \Z}(-1)^{j-1} E_{j,j+1}\right) \\
    &= H(\bx) \sum_{j \in \Z} (-1)^{j-1} e^{-x_j} E_{j,j} + H(\bx) \sum_{j \in \Z}(-1)^{j-1} E_{j,j+1} \\
    &=\sum_{i\in \Z} (-1)^{i-1}E_{i,i} +\sum_{j> i} \bigl( (-1)^{j-1} e^{x_i + \cdots + x_{j-1}} + (-1)^{j-2} e^{x_i + \cdots + x_{j-1}} \bigr)E_{i,j} \\
    &= \Delta\ ,
\end{align*}
where the third equality is obtained by shifting the index of the second summand from $j$ to $j-1$ and last equality above follows by a telescoping of terms. A similar calculation shows that $E(-\bx)\Delta H(\bx) = \Delta$. This concludes the proof, since one can easily check that $\Delta^2=\mathrm{Id}$.
\end{proof}

\begin{remark}
    Noticing that $\Delta H(\bx)\Delta$ is a left and right inverse for $E(-\bx)$, we see that it is an inverse in the semi-group of finite valued upper-triangular $\Z\times\Z$ matrices under matrix multiplication. Under this notation, we see that Lemma \ref{lem:EH_relation} translates to
    \be \label{eq:H-inverse} 
    H(\bx)=\Delta E(-\bx)^{-1}\Delta.
    \ee
\end{remark}
\begin{lemma}\label{lem:EDT_identity}
    For $\mbf X_1,\mbf X_2\in\R^{\ZN}$, we have that
    \be\label{eq:EDT_identity}
    E(-\mbf X_2)E(-\mbf X_1)=E(-D(\mbf X_1,\mbf X_2))E(-T(\mbf X_1, \mbf X_2)).
    \ee
\end{lemma}
\begin{proof}
By direct computation and applying Lemma \ref{lem:DT_identities} in the third equality below, we see that
\begin{align*}
    &\quad \, E(-D(\mbf X_1,\mbf X_2))E(-T(\mbf X_1,\mbf X_2)) \\
    &=\left(\sum_{i\in\Z} e^{-D_{i}(\mbf X_1,\mbf X_2)}E_{i,i}+\sum_{i\in\Z}E_{i,i+1}\right) \left(\sum_{i\in\Z} e^{-T_{i}(\mbf X_1,\mbf X_2)}E_{i,i}+\sum_{i\in\Z}E_{i,i+1}\right)\\
    &=\sum_{i\in\Z} e^{-(D_{i}(\mbf X_1,\mbf X_2)+T_{i}(\mbf X_1,\mbf X_2))} E_{i,i}+\sum_{i\in\Z} \left(e^{-D_{i}(\mbf X_1,\mbf X_2)}+e^{-T_{i+1}(\mbf X_1,\mbf X_2)}\right) E_{i,i+1}+\sum_{i\in \Z} E_{i,i+2}\\
    &=\sum_{i\in\Z} e^{-(X_{1,i}+X_{2,i})} E_{i,i}+\sum_{i\in\Z}\left(e^{-X_{1,i+1}}+e^{-X_{2,i}}\right) E_{i,i+1}+\sum_{i\in \Z} E_{i,i+2}\\
    &= \left(\sum_{i\in\Z} e^{-X_{2,i}}E_{i,i}+\sum_{i\in\Z}E_{i,i+1}\right) \left(\sum_{i\in\Z} e^{-X_{1,i}}E_{i,i}+\sum_{i\in\Z}E_{i,i+1}\right) =E(-\mbf X_2)E(-\mbf X_1). \qedhere
\end{align*}
\end{proof}

The following is the key step that will allow us to prove both Theorem \ref{thm:braid} and Theorem \ref{thm:polymer_invariance}. 
\begin{proposition} \label{prop:H_relation}
For $\mbf X_1,\mbf X_2 \in \R^{\ZN}$, we have
\be \label{eq:H_relation}
H(\mbf X_1)H(\mbf X_2)=H(T(\mbf X_1,\mbf X_2))H(D(\mbf X_1,\mbf X_2)).
\ee
\end{proposition}
\begin{proof}
First, we justify that this matrix product has finite non-negative real entries. Because the outputs of the matrix map $H$ (as well as matrix map $E$) are upper triangular, each individual index in the product consists of a finite sum of products of elements in the respective $H$ matrices. While each row and column has infinitely many non-zero terms, there are only a finite number of non-zero terms in each dot product. Further, the resulting matrix is upper triangular, so by repeating this argument, we can see that an arbitrary product of $H$ (or $E$) matrices has finite non-negative real entries.

With this in mind, the proof follows from Lemma \ref{lem:EH_relation}, since:
        \begin{align*}
            H(\mbf X_1)H(\mbf X_2)&\overset{\eqref{eq:EH_relation}}{=}(\Delta E(-\mbf X_1)^{-1} \Delta)(\Delta E(-\mbf X_2)^{-1}\Delta)\\
            &=\Delta E(-\mbf X_1)^{-1}E(-\mbf X_2)^{-1}\Delta\\
            &=\Delta \bigl(E(-\mbf X_2)E(-\mbf X_1)\bigr)^{-1}\Delta\\
            &\overset{\eqref{eq:EDT_identity}}{=}\Delta \bigl(E(-D(\mbf X_1,\mbf X_2))E(-T(\mbf X_1,\mbf X_2))\bigr)^{-1}\Delta\\
            &=(\Delta E(-T(\mbf X_1,\mbf X_2))^{-1}\Delta)(\Delta E(-D(\mbf X_1,\mbf X_2))^{-1}\Delta)\\
            &=H(T(\mbf X_1,\mbf X_2))H(D(\mbf X_1,\mbf X_2)). \qedhere
        \end{align*}
\end{proof}

\subsection{Proof of Theorem \ref{thm:braid}} \label{subsec:braid}

As observed previously, Item \ref{it:commute} is immediate from the definition. We prove Items \ref{it:involution} and \ref{it:braid} separately. 
\subsubsection{Proof of Item \ref{it:involution}:}
Our goal is to show that
\be\label{eq:P-involution}
    \Pcal (\Pcal(\mbf X_1,\mbf X_2)) = (T(T(\mbf X_1,\mbf X_2),D(\mbf X_1,\mbf X_2)),D(T(\mbf X_1,\mbf X_2),D(\mbf X_1,\mbf X_2))) = (\mbf X_1,\mbf X_2).
\ee
Firstly, with the convention $Y_\ell = X_{1,\ell+1}- X_{2,\ell}$ and $Y'_\ell=T_{\ell+1}(\mbf X_1,\mbf X_2)-D_\ell(\mbf X_1,\mbf X_2)$, notice that
\[
\begin{aligned}
    T_i(T(\mbf X_1,\mbf X_2),D(\mbf X_1,\mbf X_2)) &= D_{i-1}(\mbf X_1,\mbf X_2) +\log \Biggl(\f{\sum_{\jZ} e^{Y'_{[i-1,j]}}}{\sum_{\jZ}e^{Y'_{[i,j]}}}   \Biggr)\\
    &=X_{1,i} + \log\Biggl(\f{\sum_{\jZ} e^{Y_{(i-1,j]}}}{\sum_{\jZ}e^{Y_{(i-2,j]}}}\Biggr)+\log \Biggl(\f{\sum_{\jZ} e^{Y'_{[i-1,j]}}}{\sum_{\jZ}e^{Y'_{[i,j]}}}   \Biggr),
\end{aligned}
\]
and this equals $X_{1,i}$ if and only if 
\[\f{\sum_{\jZ} e^{Y'_{[i,j]}}}{\sum_{\jZ}e^{Y_{(i-1,j]}}}= \f{\sum_{\jZ} e^{Y'_{[i-1,j]}}}{\sum_{\jZ}e^{Y_{(i-2,j]}}}.
\]
Hence, it is sufficient to prove that the quantity $\f{\sum_{\jZ} e^{Y'_{[i,j]}}}{\sum_{\jZ}e^{Y_{(i-1,j]}}}$ is the same for all $i\in \Z_N$. To show this, first observe that 
\be
\begin{aligned}
    e^{Y'_i} &= e^{X_{2,i}-X_{1,i+1}} \Biggl(\f{\sum_{\ell\in\Z_N} e^{Y_{[i,\ell]}}}{\sum_{\ell\in\Z_N}e^{Y_{[i+1,\ell]}}}   \Biggr)\Biggl(\f{\sum_{\ell\in\Z_N} e^{Y_{(i-1,\ell]}}}{\sum_{\ell\in\Z_N}e^{Y_{(i,\ell]}}}   \Biggr) \\
    \Longrightarrow e^{Y'_{[i,j]}} &= e^{-Y_{[i,j]}}\Biggl(\f{\sum_{\ell\in\Z_N} e^{Y_{[i,\ell]}}}{\sum_{\ell\in\Z_N}e^{Y_{[j+1,\ell]}}}   \Biggr) \Biggl(\f{\sum_{\ell\in\Z_N} e^{Y_{(i-1,\ell]}}}{\sum_{\ell\in\Z_N}e^{Y_{(j,\ell]}}}   \Biggr)\\
    \Longrightarrow \sum_{\jZ} e^{Y'_{[i,j]}} &=\sum_{\jZ}e^{-Y_{(i,j]}}\Biggl(\f{\sum_{\ell\in\Z_N} e^{Y_{(i,\ell]}}}{\sum_{\ell\in\Z_N}e^{Y_{[j+1,\ell]}}}   \Biggr) \Biggl(\f{\sum_{\ell\in\Z_N} e^{Y_{(i-1,\ell]}}}{\sum_{\ell\in\Z_N}e^{Y_{(j,\ell]}}}   \Biggr)\\
    \Longrightarrow \f{\sum_{\jZ} e^{Y'_{[i,j]}}}{\sum_{\jZ}e^{Y_{(i-1,j]}}}&= \Biggl(\sum_{\ell \in \Z_N} e^{Y_{(i,\ell]}}\Biggr) \sum_{j \in \Z_N} \f{e^{-Y_{(i,j]}}}{\sum_{\ell\in\Z_N}e^{Y_{[j+1,\ell]}} \sum_{\ell\in\Z_N}e^{Y_{(j,\ell]}} }.
\end{aligned}
\ee
Now, we note that $\sum_{\ell\in\Z_N}e^{Y_{[j+1,\ell]}}-\sum_{\ell\in\Z_N}e^{Y_{(j,\ell]}} = e^{\mathfrak{s}(Y)}-1$ for each $j\in\ZN$. Then, multiplying and dividing by this term, we obtain
\be\label{eq:sum-}
\begin{aligned}
    \f{\sum_{\jZ} e^{Y'_{[i,j]}}}{\sum_{\jZ}e^{Y_{(i-1,j]}}}&=\Biggl(\f{\sum_{\ell \in \Z_N} e^{Y_{(i,\ell]}}}{e^{\mathfrak{s}(Y)}-1}\Biggr) \sum_{j \in \Z_N} \f{e^{-Y_{(i,j]}}\bigl(\sum_{\ell\in\Z_N}e^{Y_{[j+1,\ell]}} -\sum_{\ell\in\Z_N}e^{Y_{(j,\ell]}}\bigr)}{\sum_{\ell\in\Z_N}e^{Y_{[j+1,\ell]}} \sum_{\ell\in\Z_N}e^{Y_{(j,\ell]}} }\\
    &=\Biggl(\f{\sum_{\ell \in \Z_N} e^{Y_{(i,\ell]}}}{e^{\mathfrak{s}(Y)}-1}\Biggr) \sum_{j \in \Z_N} \Biggl(\f{e^{-Y_{(i,j]}}}{ \sum_{\ell\in\Z_N}e^{Y_{(j,\ell]}}}-\f{e^{-Y_{(i,j]}-Y_{j+1}}}{ \sum_{\ell\in\Z_N}e^{Y_{(j+1,\ell]}}}\Biggr)\\
    &=\Biggl(\f{\sum_{\ell \in \Z_N} e^{Y_{(i,\ell]}}}{e^{\mathfrak{s}(Y)}-1}\Biggr) \Biggl(\f{1}{ \sum_{\ell\in\Z_N}e^{Y_{(i,\ell]}}}-\f{e^{-\mathfrak{s}(Y)}}{ \sum_{\ell\in\Z_N}e^{Y_{(i,\ell]}}}\Biggr)=  e^{-\mathfrak{s}(Y)},
\end{aligned}
\ee
where in the penultimate step, we start the sum at $j = i$, note that the sum is telescoping, and use the fact that 
\[
Y_{(i,j]}+Y_{j+1}=\begin{cases}Y_{(i,j+1]},&\text{if }j\not\equiv i-1\\\mathfrak{s}(Y),&\text{if }j\equiv i-1\end{cases}.
\]
Our final expression does not depend on $i$, as we wanted to prove. Similarly, we notice that
\be
\begin{aligned}
    D_i(T(\mbf X_1,\mbf X_2),D(\mbf X_1,\mbf X_2)) &= T_{i+1}(\mbf X_1,\mbf X_2) +\log \Biggl(\f{\sum_{\jZ} e^{Y'_{(i,j]}}}{\sum_{\jZ}e^{Y'_{(i-1,j]}}}   \Biggr)\\
    &=X_{2,i} + \log \Biggl(\f{\sum_{\jZ} e^{Y_{[i,j]}}}{\sum_{\jZ}e^{Y_{[i+1,j]}}}   \Biggr) + \log \Biggl(\f{\sum_{\jZ} e^{Y'_{(i,j]}}}{\sum_{\jZ}e^{Y'_{(i-1,j]}}}   \Biggr),
\end{aligned}
\ee
and this equals $X_{2,i}$ if and only if 
\[\f{\sum_{\jZ} e^{Y'_{(i,j]}}}{\sum_{\jZ}e^{Y_{[i+1,j]}}}= \f{\sum_{\jZ} e^{Y'_{[i-1,j]}}}{\sum_{\jZ}e^{Y_{(i-2,j]}}}.
\]
Hence, it is sufficient to prove that the quantity $\f{\sum_{\jZ} e^{Y'_{(i,j]}}}{\sum_{\jZ}e^{Y_{[i+1,j]}}}$ is the same for all $i\in \Z_N$. From our previous observations, we know that
\be
\begin{aligned}
    e^{Y'_{(i,j]}} &= e^{-Y_{(i,j]}}\Biggl(\f{\sum_{\ell\in\Z_N} e^{Y_{[i+1,\ell]}}}{\sum_{\ell\in\Z_N}e^{Y_{[j+1,\ell]}}}   \Biggr) \Biggl(\f{\sum_{\ell\in\Z_N} e^{Y_{(i,\ell]}}}{\sum_{\ell\in\Z_N}e^{Y_{(j,\ell]}}}   \Biggr)\\
    \Longrightarrow \f{\sum_{\jZ} e^{Y'_{(i,j]}}}{\sum_{\jZ}e^{Y_{[i+1,j]}}}&= \Biggl(\sum_{\ell \in \Z_N} e^{Y_{(i,\ell]}}\Biggr) \sum_{j \in \Z_N} \f{e^{-Y_{(i,j]}}}{\sum_{\ell\in\Z_N}e^{Y_{[j+1,\ell]}} \sum_{\ell\in\Z_N}e^{Y_{(j,\ell]}} },
\end{aligned}
\ee
which we have already shown to be equal to $e^{-\mathfrak{s}(Y)}$, which does not depend on $i$. This completes the proof.

\subsubsection{Proof of Item \ref{it:braid}} Notice that the relation $\Pcal_k \Pcal_{k+1} \Pcal_k = \Pcal_{k+1} \Pcal_k \Pcal_{k+1}$ translates into the following three equalities: 
\begin{align}
    \label{eq:braid-1}
    &T\Bigl(T(\mbf X_{k},\mbf X_{k+1}),T\bigl(D(\mbf X_{k},\mbf X_{k+1}),\mbf X_{k+2}\bigr)\biggr) =T\Bigl(\mbf X_{k},T(\mbf X_{k+1},\mbf X_{k+2})\Bigr),\\ \label{eq:braid-2}
    &D\Bigl(T(\mbf X_{k},\mbf X_{k+1}),T\bigl(D(\mbf X_{k},\mbf X_{k+1}),\mbf X_{k+2}\bigr)\Bigr) = T\Bigl(D\bigl(\mbf X_{k},T(\mbf X_{k+1},\mbf X_{k+2})\bigr),D(\mbf X_{k+1},\mbf X_{k+2})\Bigr),\\
    \label{eq:braid-3}
    &D\Bigl(D(\mbf X_{k},\mbf X_{k+1}),\mbf X_{k+2}\Bigr) = D\Bigl(D\bigl(\mbf X_{k},T(\mbf X_{k+1},\mbf X_{k+2})\bigr),D(\mbf X_{k+1},\mbf X_{k+2})\Bigr).
\end{align}
The last of these, equation \eqref{eq:braid-3} was proved previously as \cite[Proposition 5.13]{corwin2024periodicpitmantransformsjointly}, recorded in Lemma \ref{lem:braid-original} of the appendix. We now show that demonstrate that \eqref{eq:braid-3} along with Proposition \ref{prop:H_relation} implies both \eqref{eq:braid-1} and \eqref{eq:braid-2}. The main tool will be the following lemma:
\begin{lemma}\label{lem:sum-partition-inv}
    For $\mbf W=(\mbf W_i)_{i\in\Z}, \mbf W'=(\mbf W'_i)_{i\in\Z}\in (\R^{\ZN})^{\Z}$ satisfying $\vecsum(\mbf W_1)=\vecsum(\mbf W'_1)$ we have that $H(\mbf W_1)H(\mbf W_2)=H(\mbf W'_1)H(\mbf W_2')$ if and only if $\mbf W_1=\mbf W'_1$ and $\mbf W_2=\mbf W_2'$.
\end{lemma}
Before proving Lemma \ref{lem:sum-partition-inv}, we first use  Lemmas \ref{lem:sum-partition-inv} and \ref{lem:braid-original} to complete the proof of Theorem \ref{thm:braid} \ref{it:braid}:
\begin{proof}[Proof of Theorem \ref{thm:braid}\ref{it:braid}.] We apply Proposition \ref{prop:H_relation} three times: first to the last two terms, then the first two, then the last two again:
\begin{align}
&H\bigl(\mbf X_{k}\bigr)H\bigl(\mbf X_{k+1}\bigr)H\bigl(\mbf X_{k+2}\bigr)=\nonumber \\
    &\overset{\eqref{eq:H_relation}}{=}H\bigl(\mbf X_{k}\bigr)H\bigl(T(\mbf X_{k+1},\mbf X_{k+2})\bigr)H\bigl( D(\mbf X_{k+1}, \mbf X_{k+2})\bigr)\nonumber \\
    &\overset{\eqref{eq:H_relation}}{=}H\Bigl(T\bigl(\mbf X_{k},T(\mbf X_{k+1}, \mbf X_{k+2})\bigr)\Bigr) H\Bigl(D\bigl(\mbf X_{k},T(\mbf X_{k+1}, \mbf X_{k+2})\bigr)\Bigr)H\Bigl(D(\mbf X_{k+1}, \mbf X_{k+2})\Bigr)\nonumber \\
    &\overset{\eqref{eq:H_relation}}{=}H\Bigl(T\bigl(\mbf X_{k},T(\mbf X_{k+1}, \mbf X_{k+2})\bigr))\Bigr) H\Bigl(T\bigl(D\bigl(\mbf X_{k},T(\mbf X_{k+1}, \mbf X_{k+2})\bigr),D(\mbf X_{k+1}, \mbf X_{k+2})\bigr)\Bigr)\nonumber  \\ &\qquad H\Bigl(D\bigl(D\bigl(\mbf X_{k},T(\mbf X_{k+1}, \mbf X_{k+2})\bigr),D(\mbf X_{k+1}, \mbf X_{k+2})\bigr)\Bigr). \label{eq:P_k+1P_kP_k+1}
\end{align}
Alternatively, we apply Proposition \ref{prop:H_relation} to the first two terms, then the last two terms, then again the first two terms to get
\begin{align}
&H\bigl(\mbf X_{k}\bigr)H\bigl(\mbf X_{k+1}\bigr)H\bigl(\mbf X_{k+2}\bigr)=\nonumber \\
    &\overset{\eqref{eq:H_relation}}{=}H\bigl(T(\mbf X_{k}, \mbf X_{k+1})\bigr)H\bigl(D(\mbf X_{k},\mbf X_{k+1})\bigr)H\bigl(\mbf X_{k+2}\bigr)\nonumber \\
    &\overset{\eqref{eq:H_relation}}{=}H\Bigl(T(\mbf X_{k}, \mbf X_{k+1})\Bigr)H\Bigl(T\bigl(D(\mbf X_{k},\mbf X_{k+1}),\mbf X_{k+2}\bigr)\Bigr)H\Bigl(D\bigl(D(\mbf X_{k},\mbf X_{k+1}),\mbf X_{k+2}\bigr)\Bigr) \nonumber \\
    &\overset{\eqref{eq:H_relation}}{=}H\Bigl(T\bigl(T(\mbf X_{k}, \mbf X_{k+1}),T\bigl(D(\mbf X_{k},\mbf X_{k+1}),\mbf X_{k+2}\bigr)\big)\Bigr)\nonumber \\
    &\qquad H\Bigl(D\bigl(T(\mbf X_{k}, \mbf X_{k+1}),T\bigl(D(\mbf X_{k},\mbf X_{k+1}),\mbf X_{k+2}\bigr)\big)\Bigr) H\Bigl(D\bigl(D(\mbf X_{k},\mbf X_{k+1}),\mbf X_{k+2}\bigr)\Bigr). \label{eq:P_kP_k+1P_k}
\end{align}
From Lemma \ref{lem:braid-original}, we know that the last terms in \eqref{eq:P_k+1P_kP_k+1} and \eqref{eq:P_kP_k+1P_k} are equal. Because $H$ is invertible (see Lemma \ref{lem:EH_relation}), we can equate both equations and cancel out the last term, obtaining that:
\begin{align*}
    &H\Bigl(T\bigl(\mbf X_{k},T(\mbf X_{k+1}, \mbf X_{k+2})\bigr)\Bigr) H\Bigl(T\bigl(D\bigl(\mbf X_{k},T(\mbf X_{k+1}, \mbf X_{k+2})\bigr),D(\mbf X_{k+1}, \mbf X_{k+2})\bigr)\Bigr)\\
    &=H\Bigl(T\bigl(T(\mbf X_{k}, \mbf X_{k+1}),T\bigl(D(\mbf X_{k},\mbf X_{k+1}),\mbf X_{k+2}\bigr)\big)\Bigr)H\Bigl(D\bigl(T(\mbf X_{k}, \mbf X_{k+1}),T\bigl(D(\mbf X_{k},\mbf X_{k+1}),\mbf X_{k+2}\bigr)\big)\Bigr).
\end{align*}
By repeated application of Lemma \ref{lem:DT_ordering}, we have that
\begin{align*}
    &\vecsum(T(\mbf X_{k},T(\mbf X_{k+1},\mbf X_{k+2})))=\vecsum(\mbf X_{k+2})
    =
    \vecsum(T(T(\mbf X_{k}, \mbf X_{k+1}),T(D(\mbf X_{k},\mbf X_{k+1}),\mbf X_{k+2}))).
\end{align*}
Applying Lemma \ref{lem:sum-partition-inv}, we conclude \eqref{eq:braid-1} and \eqref{eq:braid-2}, as we wanted to prove.
\end{proof}

\begin{proof}[Proof of Lemma \ref{lem:sum-partition-inv}.]
    Immediately, the backwards direction is clear. Then, it remains to show that, if  $H(\mbf W_1)H(\mbf W_2)=H(\mbf W_1')H(\mbf W_2')$, then $\mbf W_1=\mbf W_1'$ and $\mbf W_2=\mbf W_2'$. Let us prove by backwards induction on $k=N,\ldots, 1$ that $W_{1,k}=W_{1,k}'$ and $W_{2,k}=W_{2,k}'$.
    
    We start with $k=N$ as our base case. From Lemma \ref{lem:H_partition} below, we have that
    \begin{align}
    \left(H(\mbf W_1)H(\mbf W_2)\right)_{1,N}
    &=Z^{\mbf W}(2,N\mid 1,1)\nonumber\\
    &= e^{W_{2,N}} \left(Z^{\mbf W}(1,N\mid 1,1)+Z^{\mbf W}(2,N-1\mid 1,1)\right)\nonumber\\
    &= e^{W_{2,N}}\Biggl(\prod_{i=1}^{N} e^{W_{1,i}}+\left(H(\mbf W_1)H(\mbf W_2)\right)_{1,N-1}\Biggr).\label{eq:inductionY}
    \end{align}
    Similarly, we conclude that
    \begin{align}
    \left(H(\mbf W'_1)H(\mbf W'_2)\right)_{1,N}
    &=Z^{\mbf W'}(2,N\mid 1,1)\nonumber\\
    &= e^{W'_{2,N}} \left(Z^{\mbf W'}(1,N\mid 1,1)+Z^{\mbf W'}(2,N-1\mid 1,1)\right)\nonumber\\
    &= e^{W'_{2,N}}\Biggl(\prod_{i=1}^{N} e^{W'_{1,i}}+\left(H(\mbf W'_1)H(\mbf W'_2)\right)_{1,N-1}\Biggr).\label{eq:inductionY'}
    \end{align}
    By assumption, we have $H(\mbf W_1)H(\mbf W_2)=H(\mbf W'_1)H(\mbf W_2')$ and $\prod_{i=1}^{N} e^{W_{1,i}} = e^{\vecsum(W_1)} = e^{\vecsum(W'_1)} = \prod_{i=1}^{N} e^{W'_{1,i}}$. Equating \eqref{eq:inductionY} and \eqref{eq:inductionY'} and canceling the nonzero expressions in parentheses yields $W_{2,N}=W'_{2,N}$. Next, we see that $W'_{1,N}=W_{1,N}$, since:
    \be
    (H(\mbf W_1)H(\mbf W_2))_{(N,N)}=Z^{\mbf W}(2,N\mid 1,N)=e^{W_{1,N}+W_{2,N}},
    \ee
    and
    \be
    (H(\mbf W'_1)H(\mbf W_2'))_{(N,N)}=Z^{\mbf W'}(2,N\mid 1,N)=e^{W'_{1,N}+W'_{2,N}}.
    \ee
    To finish the proof, simply notice that the proof of the above only required $\prod_{i=1}^{N} e^{W'_{1,i}}=\prod_{i=1}^{N} e^{W_{1,i}}$. Indeed, if we have 
    $
    \prod_{i = 1}^k e^{W_{1,i}} = \prod_{i = 1}^k e^{W_{2,i}}
    $
    for some $k \in \llbracket 1, N \rrbracket$, then the same proof shows that $W_{1,k} = W_{1,k}'$ and $W_{2,k}= W_{2,k}'$. This follows by backwards induction on $k$, since 
    \[
    \prod_{i = 1}^k e^{W_{1,i}} = e^{\vecsum(\mbf W_1)} \Biggl(\prod_{i= k+1}^N e^{W_{1,i}}\Biggr)^{-1}. \qedhere
    \]
\end{proof}

\section{Invariance of polymer partition functions} \label{sec:polymer_invariance}
In this section, we prove Theorem \ref{thm:polymer_invariance}. The key to our proof is a matrix encoding of polymer partition functions, analogous to that developed in \cite{noumi2002tropicalrobinsonschenstedknuthcorrespondencebirational}. In our setting, this is given by the following lemma. 
\begin{lemma}\label{lem:H_partition} Let $\mbf X=(\mbf X_i)_{i\in\Z}\in(\R^{\ZN})^{\Z}$ and $k,\ell\in \Z$ such that $\ell\ge k$, then:
\[
\Bigl(H(\mbf X_{k})H(\mbf X_{k+1})\cdots H(\mbf X_{\ell})\Bigr)_{(i,j)}=Z^{\mbf X}(\ell,j\mid k,i) \quad \forall i,j\in\Z.
\]
In other words, the partition function across any two points $(k,i)$ and $(\ell, j)$ on the directed lattice $\Z^2$ with periodic weights $\mbf X$ is given by the $(i,j)$th entry in the matrix product of the $H$ matrices.
\begin{proof}
    We prove the result by induction on $\ell\ge k$. For the base case, $\ell=k$, by definition we have that
    \[H(\mbf X_k)_{(i,j)}=\begin{cases}
        \exp\Biggl({\sum\limits_{r=i}^j}X_{(k,r)}\Biggr),\quad &j\geq i\\
        0,\ &\text{otherwise}
    \end{cases}\]
    Similarly, on a single column, we either have a unique up/right path between $(k,i)$ and $(k,j)$ picking up all the weights between $i$ and $j$ or have no path between them, depending on whether $j\ge i$. Thus, \[Z^X(k,j\mid k,i)=\begin{cases}
        \exp\Biggl({\sum\limits_{r=i}^j}X_{(k,r)}\Biggr),\quad &j\geq i\\
        0,\ &\text{otherwise}
    \end{cases}=H(\mbf X_k)_{(i,j)}\]
    This proves our base case. Now, let us assume that it holds for some $\ell$, and prove it holds for $\ell+1$. We split into two cases: $j\ge i$ and $j<i$. 
    
    For $j<i$, notice that $H$ is upper triangular, so their product is as well. Thus, the $(i,j)$-th index of our matrix product is $0$ for $j<i$. This matches up with our partition function, as there are no up/right paths whose initial point is lower than its end point.
    
    For the case $j\ge i$, we notice that every path from $(k,i)$ to $(\ell+1,j)$ consists of a path from $(k,i)$ to $(\ell,n)$ concatenated with a path from $(\ell+1,n)$ to $(\ell+1,j)$, for some $i\le n\le j$.
    Thus, 
    \begin{align*}
        Z^{\mbf X}(\ell+1,j\mid k,i)&=\sum\limits_{r=i}^j\Biggl(Z^{\mbf X}(\ell,r\mid k,i)\exp\Bigl(\sum_{s=r}^j X_{(\ell+1,s)}\Bigr)\Biggr)\\
       &=\sum\limits_{r=i}^j\Biggl(Z^{\mbf X}(\ell,r\mid k,i)\Bigl(H(\mbf X_{\ell+1})\Bigr)_{(r,j)}\Biggr).
    \end{align*}
    By our inductive hypothesis, we know that this equals 
    \[\sum\limits_{r=i}^j\Biggl(\Bigl(H(\mbf X_k)H(\mbf X_{k+1})\ldots H(\mbf X_{\ell})\Bigr)_{(i,r)}\Bigl(H(\mbf X_{\ell+1})\Bigr)_{(r,j)}\Biggr).\]
    Remembering now that $\bigl(H(\mbf X_{\ell+1})\bigr)_{(r,j)}=0$ for $r>j$, and $\bigl(H(\mbf X_k)H(\mbf X_{k+1})\ldots H(\mbf X_{\ell})\bigr)_{(i,r)}=0$ for $r<i$, we see that 
    \begin{align*}
        Z^{\mbf X}(\ell+1,j\mid k,i) &=\sum\limits_{r=-\infty}^\infty\Biggl(\Bigl(H(\mbf X_k)H(\mbf X_{k+1})\ldots H(\mbf X_{\ell})\Bigr)_{(i,r)}\Bigl(H(\mbf X_{\ell+1})\Bigr)_{(r,j)}\Biggr)\\
        &=\Bigl(H(\mbf X_k)H(\mbf X_{k+1})\ldots H(\mbf X_{\ell})H(\mbf X_{\ell+1})\Bigr)_{(i,j)}.\qedhere
    \end{align*}
\end{proof}
\end{lemma}
We now show the invariance of polymer partition function under the periodic Pitman transform. Retaining the periodic weight vectors $\mbf X = (\mbf X_r \in \R^{\ZN}: r \in \Z)$ as before, for $k \in \Z$, recall the operator $\Pcal_k: (\R^{\ZN})^\Z \to (\R^{\ZN})^\Z$ where
\[
\Pcal_k \mbf X = (\ldots, \mbf X_{k-1}, T(\mbf X_{k}, \mbf X_{k+1}), D(\mbf X_{k}, \mbf X_{k+1}), \mbf X_{k+2, \ldots}).
\]

\begin{proposition}\label{prop:multi_column}
    Fix $a,c \in \Z$ such that $a\leq c$, and let $k_1, \dots, k_m \in \Z$ such that for all $i\in\llbracket1,m\rrbracket$, $a\ne k_{i}+1$ and $c\ne k_i$. Then, for all $b,d \in \Z$,
    \[
    Z^{\Pcal_{k_m} \cdots \Pcal_{k_1} \mbf X}(c, d \mid a, b) = Z^{\mbf X}(c, d \mid a, b).
    \]
\end{proposition}

\begin{proof}
    
    We will prove this by induction on $m$. By Lemma $\ref{lem:H_partition}$, 
    $$Z^{\mbf X}(c,d \mid a,b) = \bigl(H(\mbf X_a) H(\mbf X_{a+1}) \cdots H(\mbf X_c)\bigr)_{(b,d)},$$ which we claim is preserved by the action $\Pcal_{k_m}\ldots\Pcal_{k_1}$ for any $b,d \in \Z$. In our base case $m = 1$, we split into three sub-cases: $k_1 < a-1$, $k_1 > c$, and $a\leq k_1<c$. In the first two sub-cases, $\mbf X_a\ldots\mbf X_c$ are unaffected by $\Pcal_{k_1}$, so our invariance holds trivially.

    If $a \leq k_1 < c$, then Proposition \ref{prop:H_relation} gives that
    \begin{align*}
        Z^{\mbf X}(c,d \mid a,b)&=\bigl(H(\mbf X_a) H(\mbf X_{a+1}) \ldots H(\mbf X_{k_1})H(\mbf X_{k_1+1})\ldots H(\mbf X_c)\bigr)_{(b,d)}\\
        &\overset{\eqref{eq:H_relation}}{=}\bigl(H(\mbf X_a) H(\mbf X_{a+1}) \ldots H(T(\mbf X_{k_1}, \mbf X_{k_1 + 1}))H(D(\mbf X_{k_1}, \mbf X_{k_1 + 1}))\ldots H(\mbf X_c)\bigr)_{(b,d)}\\
        &=Z^{\Pcal_{k_1}\mbf X}(c,d \mid a,b).
    \end{align*}
    
    Now, we assume by way of induction that the invariance holds for $m-1$ transforms. Again, if $k_m < a-1$ or $k_m > c$ our equivalence holds trivially. Otherwise, let $\mbf X^\prime = \Pcal_{k_{m-1}} \ldots \Pcal_{k_1} \mbf X$. Then \begin{align*}
        Z^{\mbf X}(c,d \mid a,b)&=Z^{\mbf X^\prime}(c,d \mid a,b)\\
        &=Z^{\Pcal_{k_m}\mbf X^\prime}(c,d \mid a,b)\\
        &=Z^{\Pcal_{k_m}\Pcal_{k_{m-1}} \ldots \Pcal_{k_1} \mbf X}(c,d \mid a,b).
    \end{align*}
    This concludes our proof, noting that the second equality follows from the case $m=1$.
\end{proof}

With Proposition \ref{prop:multi_column} proved, we now give one more intermediate lemma before proving Theorem \ref{thm:polymer_invariance}.

\begin{lemma} \label{lem:two_conditions_connect}
    Let $\sigma:\Z \to \Z$ be a finite permutation, decomposed as a product of nearest-neighbor transpositions $(k_m,k_{m}+1)(k_{m-1},k_{m-1} + 1) \cdots(k_1,k_1+1)$, where $k_1,\ldots,k_m \in \Z$. Assume that this decomposition has the minimal number of terms. Then, if $x \in \Z$ satisfies $\sigma(\Z_{<x}) = \Z_{<x}$, then $x \neq k_\ell + 1$ for all $1 \le \ell \le m$. Similarly, if $\sigma(\Z_{>x}) = \Z_{>x}$, then $x \neq k_\ell$ for all $1 \le \ell \le m$.
\end{lemma}
\begin{proof}
    We prove the first statement; the second has a symmetric proof. If $\sigma(\Z_{<x}) = \Z_{<x}$, then $\sigma$ can be decomposed as $\sigma = \sigma_1\sigma_2$, where $\sigma_1$ permutes only elements of $\Z_{<x}$, and $\sigma_2$ permutes only elements of $\Z_{\ge x}$. By minimality of the decomposition $\sigma = (k_m,k_{m}+1)(k_{m-1},k_{m-1} + 1) \cdots(k_1,k_1+1)$, we must have that 
    \[
    \sigma_1 = (k_{\ell_r},k_{\ell_r}+1) \cdots(k_{\ell_1},k_{\ell_1}+1),\quad\text{and}\quad \sigma_2 = (k_{p_n},k_{p_n}+1) \cdots(k_{p_1},k_{p_1}+1)
    \]
    for disjoint subsequences $k_{\ell_1},\ldots,k_{\ell_r}$ and $k_{p_1},\ldots,k_{p_n}$ with $r+n = m$. Since $\sigma_1$ only permutes elements of $\Z_{< x}$, we have that $k_{\ell_t} + 1 < x$ for $1 \le t \le r$. Since $\sigma_2$ only permutes elements of $\Z_{\ge x}$, we have $k_{p_t} + 1 > x$ for $1 \le t \le n$. Hence,
    $x \neq k_{\ell} + 1$ for $1 \le \ell \le m$.
\end{proof}

\begin{proof}[Proof of Theorem \ref{thm:polymer_invariance}.]
Let $\sigma \in \mathcal S_{\Z}$, and let $(U,V) \in \mathcal W^\sigma$ (recalling the definition of this set in \eqref{eq:Usigma_cond}). Denote the points of $U$ and $V$ as $U=\bigl((a_i,b_i)\bigr)_{i\in\llbracket1,k\rrbracket}$, $V=\bigl((c_i,d_i)\bigr)_{i\in\llbracket1,k\rrbracket}$. Let $\sigma \in \mathcal S_{\Z}$.  By assumption that $(U,V) \in \mathcal W^\sigma$,  Lemma \ref{lem:two_conditions_connect} implies that the conditions of Proposition \ref{prop:multi_column} are satisfied so that, for each $i\in\llbracket1,k\rrbracket$, 
\begin{equation} \label{eq:ZXsigmaX}
Z^{\mbf X}(c_j,d_j\mid a_i,b_i) = Z^{\mathcal P_\sigma \mbf X}(c_j,d_j\mid a_i,b_i).
\end{equation}
To extend this to the multi-path result, we use the Lindstr\"{o}m-Gessel-Viennot (LGV) Lemma (cf.~\cite[Chapter 32]{Aigner-Ziegler-book}, originally from \cite{Lindstrom-1973,Gessel-Viennot-1985}) to prove the theorem. Because LGV applies only to edge-weighted graphs, we need to transfer our vertex weights to edge weights in the following way: for each vertex $(r,s)$, split it into $(r,s)_{\text{in}}$ and $(r,s)_{\text{out}}$, with edge $\bigl((r,s)_{\text{in}},(r,s)_{\text{out}}\bigr)$ given the weight $\omega_{(r,s)} := e^{X_{(r,s)}}$, and edges $\bigl((r,s)_{\text{out}},(r, s+1)_{\text{in}}\bigr)$ and $\bigl((r,s)_{\text{out}}, (r+1,s)_{\text{in}}\bigr)$ given the weight 1. See Figure \ref{fig:LGV_split}. Any path which started at $(r,s)$ starts at $(r,s)_{\text{in}}$, while any path ending at $(r,s)$ ends at $(r,s)_{\text{out}}$. Under this graph transformation, our partition function on the pre-transformed graph equals the sum over weight products along possible paths on the transformed diagram, so we can apply LGV.
    \begin{figure}[ht]
    \centering
    \begin{tikzpicture}[-Latex, every node/.style={font=\small}]

    \coordinate (v) at (0,0);
    \filldraw (v) circle (2pt); 
    \node at (0.5,0.3) {$X_{(r,s)}$};

    \draw[->] ($(v)+(-1.5,0)$) -- (v);       
    \draw[->] ($(v)+(0,-1.5)$) -- (v);       

    \draw[->] (v) -- ++(1.2,0);              
    \draw[->] (v) -- ++(0,1.2);              

    \draw[double equal sign distance, -Implies, thick] (2.4,0.3) -- (4.3,0.3);

    \node[fill=black, circle, inner sep=1pt, label=above:{in}] (in) at (6,0) {};
    \node[fill=black, circle, inner sep=1pt, label=below:{out}] (out) at (7.5,0) {};

    \draw[->] (in) -- (out) node[midway, below] {$\omega_{(r,s)}$};

    \draw[->] ($(in)+(-1.2,0)$) -- (in) node[midway, above] {1};     
    \draw[->] ($(in)+(0,-1.2)$) -- (in) node[midway, right] {1};     

    \draw[->] (out) -- ++(1.2,0) node[midway, above] {1};            
    \draw[->] (out) -- ++(0,1.2) node[midway, right] {1};            

    \end{tikzpicture}
    \caption{Local transformation of a vertex-weighted node (left) with weight $X_{(r,s)}$ into an edge-weighted directed graph (right), where the internal edge carries weight $\omega_{(r,s)}=\exp(X_{(r,s)})$ and external edges carry weight 1. Paths in the original graph correspond to paths from ``in'' nodes to ``out'' nodes in the transformed graph, preserving path weights.}
    \label{fig:LGV_split}
    \end{figure}
    
     \noindent Define the matrix $M^{\mbf X}:=\bigl(Z^{\mbf X}(c_j,d_j\mid a_i,b_i)\bigr)_{(i,j)\in\llbracket 1, m \rrbracket}$. By LGV, we have
    \[\det M^{\mbf X} = \sum\limits_{P=(P_1\ldots P_m)\in\Pi_{U, V_\tau} \ \forall\varphi\in \mathcal S_m} \operatorname{sgn}(\varphi) \prod\limits_{(i,j)\in P} e^{X_{(i,j)}}.\]
    Since $\mathcal W^\sigma \subseteq \Psi$, we have that, by definition of $\Psi$ \eqref{def:Psi}, for $\varphi\neq \Id$, $\Pi_{U,V_\varphi}=\varnothing$. Then, because $\operatorname{sgn}(\Id)=1$,
    \begin{equation} \label{eq:det = Z}
    \det M^{\mbf X}=\sum\limits_{P=(P_1\ldots P_m)\in\Pi_{U, V}} \prod\limits_{(i,j)\in P} e^{X_{(i,j)}}=Z^{\mbf X}(V\mid U).
    \end{equation}
    By \eqref{eq:ZXsigmaX}, for each $i,j$, we have $M_{i,j}^{\mbf X} = M_{i,j}^{\Pcal_\sigma \mbf X}$. Therefore, by \eqref{eq:det = Z}, we also have
    \[
    Z^{\Pcal_\sigma\mbf X}(V\mid U) = \det M^{\Pcal_\sigma \mbf X} = \det M^{\mbf X} =  Z^{\mbf X}(V\mid U). \qedhere
    \] 
\end{proof}

\section{Burke property and permutation invariance} \label{sec:Burke_permutation}
In this section, we prove Theorems \ref{thm:per_permutation-invariance} and \ref{thm:full_per_invariance}. We first need two intermediate lemmas. In the first, we generalize the Burke property in \cite[Proposition 5.12]{corwin2024periodicpitmantransformsjointly} to an inhomogeneous environment.
\begin{proposition}[Log-inverse-gamma Burke property] \label{burke-pos-temp}
    Let $\lambda > 0$, let $a_1,a_2 \in \R$, and let $\mbf b \in \R^{\ZN}$ be a sequence such that $a_m +b_i > 0$ for $m \in \{1,2\}$ and $\iZ$. Let $\mbf X_1, \mbf X_2$ be independent sequences of independent random variables such that $X_{m,i} \sim\LIG(a_m + b_i,\lambda)$ (see Definition \ref{def:log-inv-gamma}). Then,
    \[(\mbf X_1,\mbf X_2) \deq (D(\mbf X_1, \mbf X_2),T(\mbf X_1, \mbf X_2)).
    \]
\end{proposition}

\begin{proof}
    Using Lemmas \ref{lem:DT_ordering}-\ref{lem:DT_identities} in the penultimate equality below, the joint density $\Prob(\mbf X_1\in d\mbf x_1,\mbf X_2\in d\mbf x_2)$ is proportional to
    \begin{align*}
    &\exp\Biggl(-\sum_{i\in\ZN}\Bigl((a_1+b_i) x_{1,i}+(a_2+b_i)x_{2,i}\Bigr)-\lambda\sum_{i\in\ZN}(e^{-x_{1,i}}+e^{-x_{2,i}})\Biggr)\\
    &=\exp\Biggl(-a_1\sum_{i\in\ZN}x_{1,i}-a_2\sum_{i\in\ZN}x_{2,i}-\sum_{i\in\ZN}b_i(x_{1,i}+x_{2,i})-\lambda\sum_{i\in\ZN}(e^{-x_{2,i}}+e^{-x_{1,i+1}})\Biggr)\\
    &=\exp\Biggl(-a_1\sum_{i\in\ZN}D_i(\mbf x_1,\mbf x_{2})-a_2\sum_{i\in\ZN}T_i(\mbf x_1,\mbf x_{2})-\sum_{i\in\ZN}b_i(T_i(\mbf x_1,\mbf x_{2})+D_i(\mbf x_1,\mbf x_{2}))\\
    &\qquad\qquad\qquad\qquad -\lambda\sum_{i\in\ZN}(e^{-T_{i+1}(\mbf x_1,\mbf x_{2})}+e^{-D_i(\mbf x_1,\mbf x_{2})})\Biggr)\\
    &=\exp\Biggl(-\sum_{i\in\ZN}\Bigl((a_1+b_i)D_i(\mbf x_1,\mbf x_{2})+(a_2+b_i)T_i(\mbf x_1,\mbf x_{2})\Bigr)-\lambda\sum_{i\in\ZN}(e^{-D_i(\mbf x_1,\mbf x_{2})}+e^{-T_i(\mbf x_1,\mbf x_{2})})\Biggr).
    \end{align*}
    Thus, the  density is preserved under our transformation, so all we must show is that the determinant of the Jacobian matrix of the transformation has absolute value 1. This was shown previously in \cite{corwin2024periodicpitmantransformsjointly}, recorded in the present paper as Lemma \ref{lem:jacobian}.
\end{proof}

\begin{proof}[Proof of Theorem \ref{thm:per_permutation-invariance}.]
For an arbitrary finite permutation $\sigma$,
Theorem \ref{thm:polymer_invariance} implies that
\[
Z^{\Pcal_{\sigma}\mbf X^{\mbf a,\mbf b}}(V \mid U) =Z^{\mbf X^{\mbf a,\mbf b}}(V \mid U),\quad\text{for all }(U,V) \in \mathcal W^{\sigma}.
\]
 Furthermore, by successively applying the Burke property in Proposition \ref{burke-pos-temp}, we see that each operator $\Pcal_{k}$ swaps the column parameters in columns $k$ and $k+1$, so we have that $\Pcal_{\sigma}\mbf X^{\mbf a,\mbf b} \deq \mbf X^{\sigma \mbf a,\mbf b}$. This completes the proof.
\end{proof}

\begin{proof}[Proof of Theorem \ref{thm:full_per_invariance}.]
This is proved similarly as Theorem \ref{thm:per_permutation-invariance}. Consider the independent weights in the finite grid $\llbracket -M,M \rrbracket \times \llbracket -M,M \rrbracket$, and extend these periodically. Consider first end point pairs $(U,V)$ that consist of points lying in the finite grid $\llbracket -M,M \rrbracket \times \llbracket -M,M \rrbracket$, and apply the partition function invariance in Theorem \ref{thm:polymer_invariance} and the Burke property in Proposition \ref{burke-pos-temp} to both rows and columns, with $N = 2M + 1$. The result follows by sending $M \to \infty$.   
\end{proof}

\appendix

\section{Inputs from previous results and zero-temperature analogues} 
\subsection{Additional algebraic properties}\label{sec:DT-basic-properties}
In \cite{corwin2024periodicpitmantransformsjointly}, index-wise shifted versions of our $T$ and $D$ maps (with $\mbf X_1$ and $\mbf X_2$ swapped) were denoted by $T^{N,2}$ and $D^{N,2}$ respectively. To be precise, the definition of the mappings in that paper were 
\begin{align*}
T_i^{N,2}(\mbf{\mbf{X}_1}, \mbf{\mbf{X}_2}) &= X_{1,i} + \log\Biggl(\frac{\sum_{j\in\ZN}e^{X_{2,[i,j]}-X_{1,[i,j]}}}{\sum_{j\in\ZN}e^{X_{2,[i+1,j]}-X_{1,[i+1,j]}}}\Biggr),\quad\text{and} \\
D_i^{N,2}(\mbf{\mbf{X}_1}, \mbf{\mbf{X}_2}) &= X_{2,i} + \log\Biggl(\frac{\sum_{j\in\ZN}e^{X_{2,(i,j]}-X_{1,(i,j]}}}{\sum_{j\in\ZN}e^{X_{2,(i-1,j]}-X_{1,(i-1,j]}}}\Biggr).
\end{align*}
In the present paper, we subsumed the dependence on $N$. From the definitions given in the present paper \eqref{eq:pitman-transform}, we can see that 
\be \label{eq:Tshift}
    T(\mbf X_1,\mbf X_2) = \tau^{-1} T^{N,2}(\mbf X_2,\tau \mbf X_1),\quad\text{and}\quad D(\mbf X_1,\mbf X_2) = D^{N,2}(\mbf X_2,\tau \mbf X_1),
\ee
where $\tau :\R^{\Z_N} \to \R^{\Z_N}$ is the shift map defined by $(\tau \mbf X)_i = X_{i+1}$.
With this in mind, we now cite two results from \cite{corwin2024periodicpitmantransformsjointly}, which follow from translating them into appropriate statements for the present paper via the connection in \eqref{eq:Tshift}.
\begin{lemma} \cite[Proposition 5.12]{corwin2024periodicpitmantransformsjointly} \label{lem:jacobian}
The determinant of the Jacobian of the transformation 
\[
\Pcal: \R^{\Z_N} \to \R^{\Z_N} \text{ defined by }(\mbf X_1,\mbf X_2)\to \bigl(T(\mbf X_1,\mbf X_2), D(\mbf X_1,\mbf X_2)\bigr)
\]
has absolute value $1$.
\end{lemma}

\begin{lemma}\cite[Proposition 5.13]{corwin2024periodicpitmantransformsjointly}\label{lem:braid-original} For $\mbf X_1,\mbf X_2,\mbf X_3 \in \R^{\Z_N}$,
\[D\bigl(D(\mbf X_1,\mbf X_2),\mbf X_3\bigr)=D\bigl(D(\mbf X_1,T(\mbf X_2,\mbf X_3)),D(\mbf X_2,\mbf X_3)).\]
\end{lemma}

\subsection{Proofs of the zero-temperature analogues in Theorems \ref{thm:zer_temp_braid}-\ref{thm:zer_temp_perm_invar}}
\label{appx:zero_temp}

Because polymer models have a strong connection to last-passage percolation (LPP), our algebraic results will naturally generalize to \textit{zero-temperature} equivalents, through limiting arguments. In these limits,  the $(+,\times)$ algebra is swapped for the $(\max,+)$ algebra, and our maps reflect this change. As introduced in \cite{corwin2024periodicpitmantransformsjointly}, the `correct' map to look at here is $\zD$, the uniform limit as $\beta$ goes to infinity of $\frac{1}{\beta}D(\beta \mbf X_1,\beta \mbf X_2)$, as well as $\zT$, which is reached through an identical limit. In this section, we use our algebraic results about the behavior of the polymer model under the periodic Pitman transform (as well as the aforementioned limiting arguments) to prove analogous facts about LPP under the zero-temperature periodic Pitman transform. The following intermediate results are needed.
\begin{lemma}\label{lem:unif-convergence}
    For $y_1,\ldots,y_k \in \R$, as $\beta\to\infty$, $\frac{1}{\beta}\log\Bigl(\sum\limits_{i=0}^k\exp(\beta y_i)\Bigr)$ converges to $\max\limits_i y_i$, uniformly as a function of the $y_i$.
\end{lemma}
 \begin{lemma} \label{lem:conv_to_LPP}
    For any $k\in\N$ and any pair of $k$-tuples of  ordered pairs $U=\bigl((a_i,b_i)\bigr)_{i\in\llbracket1,k\rrbracket}$, $V=\bigl((c_i,d_i)\bigr)_{i\in\llbracket1,k\rrbracket}$,  as $\beta \to \infty$,  $\frac{1}{\beta}\log\bigl(Z^{\beta,\mbf X}(V\mid U)\bigr)$ converges uniformly (as a function of $\mbf X$) to $G^{\mbf X}(V\mid U)$. 
    \begin{proof}
        This follows as a direct consequence of Lemma \ref{lem:unif-convergence}, which is itself a standard fact, since there are only finitely many multi-paths between any two sets of points.
    \end{proof}
\end{lemma}

We now prove Theorems \ref{thm:zer_temp_braid}-\ref{thm:zer_temp_perm_invar}.
    \begin{proof}[Proof of Theorem \ref{thm:zer_temp_braid}.]
 To show that $\widetilde{\Pcal}$ is an involution, we take a limit from the positive-temperature case:
\begin{align*}
    \zP(\zP(\mbf{X}_1,\mbf{X}_2)) &= \lim_{\beta\to\infty} \frac1\beta \Pcal(\beta \zP(\mbf X_1, \mbf X_2)) \\
    &=\lim_{\beta \to \infty } \frac{1}{\beta} \Pcal(\Pcal(\beta\mbf X_1, \beta \mbf X_2)) \\
    &= \lim_{\beta \to \infty} \frac{1}{\beta} (\beta \mbf X_1, \beta \mbf X_2) \\
    &= (\mbf X_1, \mbf X_2),
\end{align*}
where the penultimate equality follows from $\Pcal$ being an involution and the rest follow from the uniform convergence of $\Pcal$ via Lemma \ref{lem:unif-convergence}.

    Next, for $\beta > 0$, define the operators $\Pcal_{k}^\beta:(\R^{\ZN})^\Z \to (\R^{\ZN})^\Z$ by their action:
\be \label{Pibeta}
   \begin{aligned}
&\quad \, \Pcal_k^\beta(\ldots,\mbf X_{k-1},\mbf X_{k}, \mbf X_{k+1}, \mbf X_{k+2},\ldots)  \\
&= (\ldots,\mbf X_{k-1},\frac{1}{\beta}T(\beta\mbf X_{k},\beta\mbf X_{k+1}),\frac{1}{\beta}D(\beta\mbf X_{k},\beta\mbf X_{k+1}),\mbf X_{k+2},\ldots).
\end{aligned}
\ee
As an immediate consequence of Theorem \ref{thm:braid}\ref{it:braid},     
        we see that $\Pcal_{k}^\beta \Pcal_{k+1}^\beta \Pcal_{k}^\beta=\Pcal_{k+1}^\beta \Pcal_{k}^\beta \Pcal_{k+1}^\beta$. By Lemma \ref{lem:unif-convergence}, we see that $\Pcal_{k}^\beta$ converges, as $\beta \to \infty$, to $\zP_k$ (uniformly as a function of the input). Then, we have
        \[
\zP_{k}\zP_{k+1}\zP_{k}=\lim\limits_{\beta\to\infty}\Pcal_{k}^\beta \Pcal_{k+1}^\beta \Pcal_{k}^\beta=\lim\limits_{\beta\to\infty}\Pcal_{k+1}^\beta \Pcal_{k}^\beta \Pcal_{k+1}^\beta=\zP_{k+1}\zP_{k}\zP_{k+1}. \qedhere
        \]
        \end{proof}
\begin{proof}[Proof of Theorem \ref{thm:LPP_invar}.] As in Theorem \ref{thm:polymer_invariance}, we take $\sigma \in \mathcal S_{\Z}$ and $U=\bigl(a_i,b_i)_{i\in\llbracket1,m\rrbracket}\bigr)$, and $V=\bigl((c_i,d_i)_{i\in\llbracket1,m\rrbracket}\bigr)$ satisfying $(U,V)\in \mathcal W^\sigma$. To prove Theorem \ref{thm:LPP_invar}, we must show that $G^{\mbf X}(V\mid U)=G^{\zP_\sigma \mbf X}(V\mid U)$.
   By Theorem \ref{thm:polymer_invariance} and definition of the polymer model with  parameter $\beta > 0$ \eqref{eq:multipath_partition_function}, we have
\be \label{eq:Zpbeta} Z^{\beta,\mbf X}(V\mid U)=Z^{\beta\mbf X}(V\mid U)=Z^{\Pcal_{\sigma}(\beta \mbf X)}(V\mid U) = Z^{\beta(\frac1\beta\Pcal_{\sigma}(\beta \mbf X))}(V\mid U) = Z^{\beta,\Pcal_{\sigma}^{\beta}(\mbf X)}(V\mid U),\ee 
where in the last step, we simply define $\Pcal_{\sigma}^{\beta}(\mbf X) := \f{1}{\beta} \Pcal_\sigma(\beta \mbf X)$.

As mentioned in the proof of Theorem \ref{thm:zer_temp_braid}, each $\Pcal_k^\beta$ converges uniformly to $\zP_k$ as a function of the input. By writing $\Pcal_\sigma^\beta$ as a composition of these operators and using induction, we have that $\Pcal_\sigma^\beta$ converges to $\zP_\sigma$. By Lemma \ref{lem:conv_to_LPP}, we know that $\frac{1}{\beta}\log\bigl(Z^{\beta,\mbf X}(V\mid U)\bigr)$ converges to $G^\mbf X(V\mid U)$ uniformly as a function of $\mbf X$. Hence, using \eqref{eq:Zpbeta} in the middle equality below, we have
    \[
    G^\mbf X(V\mid U)=\lim_{\beta \to \infty} \f{1}{\beta } \log Z^{\beta, \mbf X}(V \mid U) =  \lim_{\beta \to \infty} \f{1}{\beta } \log Z^{\beta, \Pcal_\sigma^\beta(\mbf X)}(V \mid U)  = G^{\zP_{\sigma}\mbf X}(V\mid U). \qedhere
    \]
\end{proof}

Before proving Theorem \ref{thm:zer_temp_perm_invar}, we first need prove that $\widetilde{\Pcal}$, restricted in its domain and range to nonnegative entries, is an involution. Then we can find analogous Burke properties in the geometric and exponential cases.

\begin{lemma} \label{W-bijection}
$\zP|_{\R^{\ZN}_{\geq 0} \times \R^{\ZN}_{\geq 0}}$ is an involution on $\R^{\ZN}_{\geq 0} \times \R^{\ZN}_{\geq 0}$.
\end{lemma}
\begin{proof}
Since $\zP$ is an involution on $\R^{\Z_N} \times \R^{\Z_N}$ by Theorem \ref{thm:zer_temp_braid}, it suffices to show that $(\mbf X_1,\mbf X_2) \in \R^{\ZN}_{\geq 0}\times \R^{\ZN}_{\geq 0}$ implies $\zT(\mbf X_1,\mbf X_2) \in \R_{\ge 0}^{\Z_N}$ and $\zD(\mbf X_1, \mbf X_2) \in \R^{\ZN}_{\geq 0}$. Fix any $i \in \Z_N$, and take $\mbf X_1,\mbf X_2 \in \R^{\ZN}_{\geq 0}$. Set $Y_i =X_{1,i+1} - X_{2,i}$ for all $\iZ$. We observe that 

\begin{align*}
    Y_{(i-1, j]}=\begin{cases}Y_i+ Y_{(i, j]},&\text{ if }\jZ \setminus \{i-1\}\\0,&\text{ if }j=i-1 \end{cases}
\end{align*}

It follows that:
\begin{align*}
    \max_{j\in\ZN} Y_{(i-1, j]} &= 0\vee\left(Y_i+\max_{j\in\ZN\setminus \{i-1\}}Y_{(i,j]}\right)\\
    &\le 0\vee\left(Y_i+\max_{j\in\ZN}Y_{(i,j]}\right)\\
    &\le 0 \vee \left(X_{1,i+1}+\max_{j\in\ZN}Y_{(i,j]}\right) \\
    &\leq X_{1,i+1}+\max_{j\in\ZN}Y_{(i,j]}.
\end{align*}

Similarly, we have that $Y_{[i, j]}=\begin{cases}-Y_{i-1}+ Y_{[i-1, j]},&\text{ if }\jZ \setminus \{i-1\} \\Y_{[i-1,i-2]},&\text{ if }j=i-1 \end{cases}$

Therefore, we know
\begin{align*}
    \max_{j\in\ZN} Y_{[i, j]} &= Y_{[i-1,i-2]}\vee\left(-Y_{i-1}+\max_{j\in\ZN\setminus \{i-1\}}Y_{[i-1,j]}\right)\\
    &\le \max_{j\in\ZN}Y_{[i-1,j]}\vee\left(X_{2,i-1}+\max_{j\in\ZN}Y_{[i-1,j]}\right)\\
    &\le X_{2,i-1}+\max_{j\in\ZN}Y_{[i-1,j]}.
\end{align*}

By the definitions in \eqref{eq:zero_temp_Pitman}, we see that $\zT_i(\mathbf{X}_1, \mathbf{X}_2)\ge 0$ and $ \zD_i(\mathbf{X}_1, \mathbf{X}_2) \geq 0$ for all $\iZ$. 
\end{proof}

        \begin{proposition}[Geometric Burke property] \label{burke_geom}
    Take $\mbf{X}_1, \mbf{X}_2 \in \Z_{\geq 0}^{\ZN}$ to be independent sequences of independent random variables such that $X_{1, i} \sim \Geom(p_1q_i)$ and $X_{2, i} \sim \Geom(p_2q_i)$ for all $\iZ$, where $p_m q_i \in (0,1)$ for $m \in \{1,2\}$ and $i \in \Z_N$. Then,
    \[ (\mbf{X}_1, \mbf{X}_2) \deq (\zD(\mbf{X}_1, \mbf{X}_2), \zT(\mbf{X}_1, \mbf{X}_2)). \]
\end{proposition}
\begin{proof}

Letting $\mbf{k}_1,\mbf{k}_2\in \Z_{\ge 0}^{N}$ where $k_{m,i}$ is the $i$th component of $\mbf{k}_m$, we have that
\begin{align*}
\Prob(\mbf{X}_1=\mbf{k}_1,\mbf{X}_2=\mbf{k}_2) &= \Prob\left(X_{1,1}=k_{1,1}\right) \cdots \Prob \left(X_{1,N}=k_{1,N}\right)\Prob \left(X_{2,1}=k_{2,1}\right)\cdots \Prob \left(X_{2,N}=k_{2,N}\right)\\
    &=\left(\prod\limits_{i = 1}^N (1-p_1 q_i) (1-p_2 q_i) \right ) p_1^{\vecsum(\mbf{k}_1)} p_2^{\vecsum(\mbf{k}_2)} \prod_{i=1}^{N} q_i^{k_{1, i} + k_{2, i}}  \\
    &= \left(\prod\limits_{i = 1}^N (1-p_1 q_i) (1-p_2 q_i) \right ) p_1^{\vecsum(\zD(\mbf{k}_1,\mbf{k}_2))} p_2^{\vecsum(\zT(\mbf{k}_1,\mbf{k}_2))} \prod_{i=1}^{N} q_i^{\zD_i(\mbf{k}_1, \mbf{k}_2) + \zT_i(\mbf{k}_1, \mbf{k}_2)}\\
    &=\prod\limits_{i = 1}^N(1-p_1q_i)(p_1q_i)^{\zD_{i}(\mbf{k}_1, \mbf{k}_2)} \prod\limits_{i = 1}^N(1-p_2q_i)(p_2q_i)^{\zT_{i}(\mbf{k}_1, \mbf{k}_2)} \\
    &= \Prob(\mbf{X}_1 = \zD(\mbf{k}_1, \mbf{k}_2), \mbf{X}_2 = \zT(\mbf{k}_1, \mbf{k}_2)) \\
    &= \Prob(\zD (\mbf{X}_1, \mbf{X}_2) = \mbf{k}_1, \zT(\mbf{X}_1, \mbf{X}_2) = \mbf{k}_2)
\end{align*}
where the first equality follows from the independence of $\mbf X_1,\mbf X_2$ and $X_{m,i},X_{m,j}$ for $m \in \{1,2\}$ and $i,\jZ$ where $i \neq j$; the second equality follows from our definition of the geometric distribution; the third equality follows from Lemmas \ref{lem:DT_ordering} and \ref{lem:DT_identities}\ref{sum-Di-Ti}; and the final equality follows from the bijection in Lemma \ref{W-bijection}.
\end{proof}

\begin{proposition}[Exponential Burke property] \label{burke_exp}
    Let $\mbf{X}_1, \mbf{X}_2 \in \Z_{\geq 0}^{\ZN}$  be independent sequences of independent random variables such that $X_{1, i} \sim \Exp(a_1  + b_i)$ and $X_{2, i} \sim \Exp(a_2 + b_i)$ for all $\iZ$ where $a_m + b_i > 0$ for $m \in \{1,2\}$ and $\iZ$. Then,

$$(\mbf{X}_1,\mbf{X}_2) \deq (\zD(\mbf{X}_1,\mbf{X}_2),\zT(\mbf{X}_1,\mbf{X}_2)).$$
\end{proposition}

\begin{proof}
 For $m \in \{1,2\}$ and $\iZ$, define 
\begin{align*}
    X_{m,i}' \sim \Geom\Biggl(\left(1-\frac{a_m }{n}\right)\left(1 - \f{b_i}{n}\right)\Biggr).
\end{align*}
Then, since $\left(1-\frac{a_m }{n}\right)\left(1 - \f{b_i}{n}\right) = 1 - \f{a_m + b_i}{n} + O(n^{-2})$, we have that
\begin{align*}
    \lim\limits_{n \to \infty} \frac{X_{m,i}'}{n} \deq X_{m,i} \sim \Exp(a_m + b_i), \quad\text{for }m \in \{1,2\},\text{ and }\iZ
\end{align*}
Observe from the definitions of $\zT$ and $\zD$ \eqref{eq:zero_temp_Pitman} that 
\[
\f{1}{n}\zT(\mbf X_1',\mbf X_2') = \zT\Bigl(\f{\mbf X_1'}{n},\f{\mbf X_2'}{n}\Bigr),\quad\text{and}\quad  \f{1}{n}\zD(\mbf X_1',\mbf X_2') = \zD\Bigl(\f{\mbf X_1'}{n},\f{\mbf X_2'}{n}\Bigr).
\]
By Proposition \ref{burke_geom}, we know that $(\mbf{X}_1',\mbf X_2') \deq (\zD(\mbf{X}_1',\mbf{X}_2'),\zT(\mbf{X}_1',\mbf{X}_2'))$. Since $\max$, addition, and subtraction are continuous, $\zD$ and $\zT$ are continuous maps.  We then have that  
\[
\Bigl(\zD(\mbf X_1,\mbf X_2),\zT(\mbf X_1,\mbf X_2)\Bigr) \deq \lim\limits_{n \to \infty} \frac{1}{n}\Bigl(\zD(\mbf X_1',\mbf X_2'),\zT(\mbf X_1',\mbf X_2')\Bigr) \deq \lim\limits_{n \to \infty} \frac{1}{n} (\mbf{X}_1', \mbf{X}_2') \deq (\mbf X_1,\mbf X_2).\qedhere
\]
\end{proof}

\begin{proof}[Proof of Theorem \ref{thm:zer_temp_perm_invar}]
The proof is the same as the proof of Theorems \ref{thm:per_permutation-invariance} and \ref{thm:full_per_invariance}: replace the use of Theorem \ref{thm:polymer_invariance} with Theorem \ref{thm:LPP_invar} and replace the Burke property in Propositions \ref{burke-pos-temp} with the Burke property in Propositions \ref{burke_geom} and \ref{burke_exp}.
\end{proof}

\bibliographystyle{alpha}
\bibliography{bibliography}
\end{document}